\documentclass{amsart}
\usepackage[utf8]{inputenc}
\usepackage{bbold}
\usepackage{amsmath, amssymb, hyperref}
\usepackage{graphicx}
\usepackage{subfigure}
\usepackage{wrapfig}
\usepackage{amsthm}
\usepackage{mathtools}
\usepackage{color}
\usepackage{float}
\usepackage{tikz}
\usepackage{braket}
\usepackage{bm}
\usetikzlibrary{cd}
\usetikzlibrary{arrows.meta}

\graphicspath{{figures/}}

\usepackage{amsmath, color}
\usepackage{wrapfig}
\usepackage{amsthm}
\usepackage[margin=1in]{geometry}
\usepackage[capitalize, nameinlink, noabbrev, nosort]{cleveref}
\usepackage{comment}
\usepackage{cases}
\usepackage{tabularx, booktabs}
\usepackage{subcaption}

\newtheorem{theorem}{Theorem}[section]
\newtheorem{lemma}[theorem]{Lemma}
\newtheorem{proposition}[theorem]{Proposition}
\newtheorem{corollary}[theorem]{Corollary}
\newtheorem{maintheorem}{Theorem}

\theoremstyle{definition}
\newtheorem{definition}[theorem]{Definition}
\newtheorem{example}[theorem]{Example}
\theoremstyle{remark}
\newtheorem{remark}[theorem]{Remark}

\numberwithin{equation}{section}

\newcommand{\Ltil}{\widetilde{L}}
\newcommand{\Gtil}{\widetilde{G}}

\newcommand{\Tbar}{\overline{T}}
\newcommand{\Gbar}{\overline{G}}
\newcommand{\ebar}{\overline{e}}
\newcommand{\That}{\widehat{T}}
\newcommand{\ehat}{\widehat{e}}
\newcommand{\altil}{\widetilde{\alpha}}
\newcommand{\Scal}{\mathcal{S}}
\newcommand{\Pcal}{\mathcal{P}}
\newcommand{\Mcal}{\mathcal{M}}
\newcommand{\ybar}{\overline{y}}
\newcommand{\ybf}{\mathbf{y}}
\newcommand{\xbf}{\mathbf{x}}
\newcommand{\gbf}{\mathbf{g}}
\newcommand{\hbf}{\mathbf{h}}
\newcommand{\dbf}{\mathbf{d}}
\newcommand{\om}{\omega}
\newcommand{\opsign}{\operatorname{sign}}
\newcommand{\opcoc}{\operatorname{cocyc}}
\newcommand{\opcyc}{\operatorname{cyc}}

\begin{document}

\title{Kauffman bracket polynomials, perfect matchings and cluster variables}

\author{}
\address{}
\curraddr{}
\email{}
\thanks{}

\author{Weiqing Tian}
\address{School of Mathematics, Nanjing University, Nanjing 210093, China}
\curraddr{}
\email{weiqingtian07@gmail.com}
\thanks{}

\subjclass[]{}


\date{}

\dedicatory{}

\begin{abstract}
We introduce a class of links whose bracket polynomials admit an expansion over perfect matchings of a plane bipartite graph. This class includes 2-bridge links, pretzel links, and Montesinos links. Our first main result (Theorem~A) provides a partial answer to a question posed by Kauffman concerning the connection between spanning tree expansions of the Jones polynomial and the Clock Theorem. Building on Theorem~A, we apply our framework to cluster theory and prove in Theorem~B that the bracket polynomials of links in this class can be realized as specializations of the F-polynomials of certain cluster variables. Theorem~B generalizes several earlier results. We also present several applications and illustrative examples.
\end{abstract}

\maketitle

\section{Introduction}
Among the various polynomial invariants of links, the Jones polynomial and the Alexander polynomial are among the most fundamental and widely studied. In \cite{thistlethwaite1987spanning}, it was shown that the Jones polynomial admits an expansion over the spanning trees of a plane graph $G$. These spanning trees have since been used to construct combinatorial models for signed Tutte polynomials \cite{kauffman1989tutte}, Khovanov homology \cite{wehrli2008spanning,champanerkar2009spanning}, and Floer homology \cite{baldwin2012combinatorial,greene2013spanning}. On the other hand, in his foundational book \cite{kauffman1983formal}, Kauffman introduced the notion of Kauffman states for a link diagram $L$ and used the clock theorem to reformulate the Alexander polynomial as a state sum. These Kauffman states have also proven useful in the context of Floer homology \cite{kauffman2016alexander,ozsvath2018kauffman}. The clock theorem was later interpreted in terms of perfect matchings on planar graphs \cite{cohen2014kauffman,propp2002lattice}, and further generalized to more complex surfaces \cite{hine2018clock}.

The relationship between the spanning tree model and the Kauffman state model is of particular interest, as there exists a natural bijection between Kauffman states of $L$ and spanning trees of a related graph. In \cite{kauffman2006remarks}, Kauffman posed the question of whether there is a deeper connection between the Jones polynomial and the Kauffman state model. The author of \cite{cohen2012determinant} suggested that such a relationship may be established via perfect matchings, at least in the case of pretzel links.

In parallel with the development of these combinatorial frameworks, the theory of cluster algebras introduced by Fomin and Zelevinsky \cite{fomin2002cluster} has emerged as a new perspective in the study of knot invariants. Notably, Lee and Schiffler \cite{lee2019cluster} showed that the Jones polynomials of 2-bridge links can be realized as specializations of F-polynomials associated with specific cluster variables. A similar result was obtained by Nagai and Terashima \cite{nagai2020cluster} for the Alexander polynomials of 2-bridge links. The results in \cite{nagai2020cluster} were later generalized to a broader class of links in a series of works \cite{bazier2022knot,bazier2024knot}, where the Kauffman state model of the Alexander polynomial played a central role; see also \cite{meszaros2024dimer} for related developments. Meanwhile, the result of Lee and Schiffler has inspired further research, particularly on $q$-deformed rational numbers \cite{morier2020q,morier2024burau}, though its potential for extension to more general link types remains an open question.

It is perhaps not surprising that these two questions are in fact closely related. This paper aims to demonstrate that the resolution of the first question—namely, under what conditions the data encoded by the spanning tree model can also be captured by the Kauffman state model—provides the key insight for addressing the second. As a consequence, several previously known results admit generalizations to broader classes of links, and certain phenomena observed in earlier studies can now be understood more deeply.  

In this paper, we focus primarily on the Kauffman bracket polynomial, introduced by Kauffman in \cite{kauffman1987state}, which is closely related to the Jones polynomial. Our first main result is stated as follows:
\begin{maintheorem}
Let $\Ltil$ be a link diagram, and let $L$ be its corresponding link universe. If there exists a distinguished segment $s$ and an ordered labeling $l$ such that the link configuration $L_{s; l}$ is admissible, then the Kauffman bracket polynomial of $\Ltil$ admits a perfect matching expansion:
\[\left\langle \Ltil \right\rangle=\sum_{\Pcal}\prod_{h \in \Pcal} \altil (h)\big|_K,\]
where the sum is taken over all perfect matchings of the balanced overlaid checkerboard graph $G^b_{L, s}$, and each half-edge $h_{(c_j, R)}$ is assigned an activity letter $\alpha(h)$ according to the following rule:
\begin{itemize}
	\item [\textbf{--}] If $c_j$ is the crossing point with the lowest label among all the crossing points incident to $R$, then $\alpha \big(h_{(c_j, R)}\big)=L$.
	\item [\textbf{--}] Otherwise, $\alpha \big(h_{(c_j, R)}\big)=D$.
\end{itemize}
\end{maintheorem}

This theorem provides a partial answer to the first question and generalizes a main result from \cite{cohen2012determinant}. It is stated in the text as Theorem~\ref{thm:perexpan}. The notion of an \textit{admissible link configuration} in the statement refers to a specific construction in which the spanning tree model and the Kauffman state model align in a compatible way. This construction was used implicitly in \cite{cohen2012determinant}, without specifying when such compatibility is possible or how to determine it. In Section 4, we present two methods for constructing admissible link configurations and prove that certain classical families of links--such as 2-bridge links, pretzel links, and Montesinos links--admit such a structure.\\

In the context of cluster algebras, our second main result is stated as follows:
\begin{maintheorem}
Let $\Ltil$ be a link diagram, and let $L_{s;l}$ be an admissible link configuration. Suppose that $\Ltil$ is alternating and all of its crossing points have negative signs. Then there exists a cluster variable $x^{Q}_{q;t}$ in the cluster algebra $\mathcal{A}(Q_{L,s})$, such that the bracket polynomial can be expressed as a specialization of the F-polynomial $F^{Q}_{q;t}$ via the bracket polynomial specialization. More precisely, we have
	\[\left\langle \Ltil \right\rangle=\om(\Scal_{\min})\cdot F^{Q}_{q;t}(\ybf)\big|_K,\]
where $\Scal_{\min}$ denotes the minimal Kauffman state of $L_s$, $\om(\Scal_{\min})$ is its weight, and the bracket polynomial specialization is given by the following rule after orienting each transposable segment $s_j$ such that the black region lies to the left of $s_j$.
\[
y_j\big|_K=
\left\{
\begin{array}
	{ll}
	A^8, & \textup {if $s_j$ goes from the lowest crossing point $c_1$ to a higher one};\\
	-A^4, & \textup {if $s_j$ goes from a lower crossing point (other than $c_1$) to a higher one};\\
	-A^{-4}, & \textup {if $s_j$ goes from a higher crossing point to a lower one, and the latter is not active};\\
	A^{-8}, & \textup {if $s_j$ goes from a higher crossing point to a lower one, and the latter is active}.
\end{array}
\right.
\]
\end{maintheorem}

This theorem is stated as Theorem~\ref{thm:fpolyspacial} and generalizes a key result from \cite{lee2019cluster}. The term \textit{active crossing points}, as used in the statement, refers to certain distinguished crossings that play a crucial role in the construction of admissible link configurations. The specialization values given in \cite{lee2019cluster} correspond to the first part of our specialization formula. The second part, involving additional segment types, does not appear in the case of 2-bridge links, which explains why it was absent in their setting. 

The paper is organized as follows. In Section~2 we review some basic terminology and results from knot theory and graph theory. Section~3 introduces a condition called the \textit{EI-property}, which governs the compatibility between the spanning tree model and the Kauffman state model. We also define the state lattice polynomial, which serves as a preparatory step toward the specialization. In Section~4, we describe two constructions of admissible link configurations and prove the first main theorem. The inclusion of classical link families—such as 2-bridge, pretzel, and Montesinos links—is also established there. In Section~5, after recalling essential notions from cluster theory, we prove the second main theorem. Finally, some applications and illustrative examples are presented at the end of Section~5.

\section{Preliminaries}

We begin by summarizing essential definitions and known results in knot theory and graph theory, which will be used throughout the paper.

\subsection{Background on knot theory}

We begin by recalling the definition of the Kauffman bracket polynomial and describing how it relates to the Jones polynomial.

\begin{definition}
A \textit{knot} is a subset of $\mathbb{R}^3$ that is homeomorphic to a circle. A \textit{link} with $r$ components is a subset of $\mathbb{R}^3$ that is homeomorphic to a disjoint union of $r$ circles. 
\end{definition}

\begin{definition}
A \textit{link diagram} $\Ltil$ is a projection of the link onto the plane that is injective except at a finite number of double points, called \textit{crossing points}, together with \textit{over-under information} at each crossing to indicate which strand passes over and which passes under. The corresponding \textit{link universe} $L$ is obtained from $\Ltil$ by discarding the over-under information at all crossing points.
\end{definition}

\begin{definition}
The \textit{Kauffman bracket polynomial} $\langle \Ltil \rangle$ is a Laurent polynomial in the variable $A$ associated to an unoriented link diagram $\Ltil$, defined recursively by the following rules:

\begin{enumerate}
	\item \textbf{Smoothing relation:}
	\[
	\left\langle
	\begin{tikzpicture}[baseline=0.1ex,scale=0.25]
		\draw[line width=0.8pt] (-0.2,-0.2) -- (1.2,1.2);
		\draw[line width=0.8pt, white, line width=3pt] (1,0) -- (0,1);
		\draw[line width=0.8pt] (1.2,-0.2) -- (-0.2,1.2);
	\end{tikzpicture}
	\right\rangle
	=
	A \left\langle
	\begin{tikzpicture}[baseline=0.1ex,scale=0.4]
		\draw[line width=0.8pt] (0,-0.2) arc (160:20:0.5);
		\draw[line width=0.8pt] (0,0.8) arc (200:340:0.5);
	\end{tikzpicture}
	\right\rangle
	+
	A^{-1} \left\langle
	\begin{tikzpicture}[baseline=0.1ex,scale=0.4]
		\draw[line width=0.8pt] (0,-0.2) arc (-70:70:0.5);
		\draw[line width=0.8pt] (1,-0.2) arc (250:110:0.5);
	\end{tikzpicture}
	\right\rangle
	,\]
	where the diagrams represent a crossing, the $A$-smoothing, and the $B$-smoothing, respectively.
	
	\item \textbf{Stabilization:}
	\[
	\left\langle \Ltil \sqcup \bigcirc \right\rangle = (-A^2 - A^{-2}) \left\langle \Ltil \right\rangle
	,\]
	where $ \Ltil \sqcup \bigcirc $ denotes the disjoint union of $\Ltil$ with an unknotted circle.
	
	\item \textbf{Normalization:}
	\[
	\left\langle \bigcirc \right\rangle = 1
	,\]
	where \( \bigcirc \) is the trivial diagram of a single unknotted circle.
\end{enumerate}
\end{definition}

\begin{figure}[hbp]
	\centering
	\subfigure[]{
		\includegraphics[width=0.6\textwidth, height=0.04\textheight]{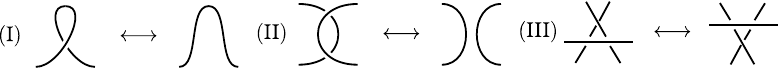}}
	\hspace{2mm}
	\subfigure[]{
		\includegraphics[width=0.18\textwidth, height=0.04\textheight]{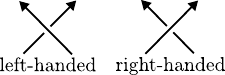}}
	\hspace{2mm}
	\subfigure[]{
		\includegraphics[width=0.08\textwidth, height=0.04\textheight]{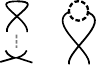}}
	\caption{(a) The three Reidemeister moves. (b) The left-handed and right-handed crossings. (c) A curl and a nugatory crossing.}
	\label{fig:Reidemeistermove}
\end{figure}

\begin{definition}
Let $D_{\Ltil}$ be an oriented diagram associated to $\Ltil$. The \textit{writhe} number $\om(D_{\Ltil})$ of the oriented diagram $D_{\Ltil}$ is defined as the number of right-handed crossings minus the number of left-handed crossings (see subfigure (b) of Figure~\ref{fig:Reidemeistermove} for an illustration of crossing types). Then the Jones polynomial $V_{\Ltil}(t)$ of $\Ltil$ is defined via the Kauffman bracket polynomial as 
\[V_{\Ltil}(t)=\Big[(-A)^{-3\om(D_{\Ltil})} \left\langle \Ltil \right\rangle \Big]\Big|_{A=t^{-1/4}},\]
where the right side denotes the evaluation of the rescaled bracket polynomial at $A=t^{-1/4}$.
\end{definition}

\begin{remark}
It is well known that two link diagrams represent isotopic links if and only if they are related by a finite sequence of Reidemeister moves (see subfigure (a) of Figure~\ref{fig:Reidemeistermove} for illustrations). The Kauffman bracket polynomial is invariant under Reidemeister moves (II) and (III). Without loss of generality, throughout this paper, we assume the link universe $L$ is connected and contains no curls or nugatory crossings (see subfigure (c) of Figure~\ref{fig:Reidemeistermove}).
\end{remark}

\begin{definition}
A link diagram $\Ltil$ is said to be \textit{alternating} if, as one traverses each component of the diagram according to its orientation, the crossings alternate between over-crossings and under-crossings. A link universe $L$ is said to be \textit{prime-like} if there does not exist a simple closed curve $S$ in the plane that intersects $L$ in exactly two points such that $L$ contains crossings both inside and outside of $S$.
\end{definition}

\subsection{Background on graph theory}

We provide some graph-theoretic background necessary for establishing the connection between link diagrams and their associated planar graphs. We then recall the definition of \textit{Tutte’s activity letter}, introduced in \cite{tutte1954contribution}, which serves as a key tool in the reformulation of the bracket polynomial.

\begin{definition}
	A \textit{graph} $G=(V,E)$ consists of a finite set of \textit{vertices} $V$, a finite set of \textit{edges} $E$, and an \textit{incidence function} that assigns to each edge an unordered pair of vertices. The \textit{endpoints} of an edge $e$ are the vertices incident to $e$.
	A \textit{loop} is an edge whose endpoints are the same vertex. An \textit{isthmus} is an edge whose deletion increases the number of connected components of the graph. 
	
	A \textit{tree} is a connected graph with no cycles. A \textit{spanning subgraph} of a graph $G$ is a subgraph that includes all the vertices of $G$. A \textit{spanning tree} is a spanning subgraph that is also a tree. A \textit{rooted spanning tree} $T(v)$ is a spanning tree with a distinguished vertex $v$, called the root of $T$.
\end{definition}

\begin{definition}
     Let $T$ be a spanning tree of $G$. An edge $e$ is said to be \textit{internal} with respect to $T$ if it is in $T$ and \textit{external} otherwise. 
     
     If $e \in T$, then removing $e$ disconnects $T$ into two components. The \textit{fundamental cocycle} of the internal edge $e$ is the set of edges $f$ such that $T \backslash \{ e\} \cup \{f\}$ is a spanning tree; equivalently, it consists of all edges that reconnect the two components of $T \backslash \{ e\}$. We denote the fundamental cocycle of $e$ with respect to $T$ by $\opcoc(e, T)$. 
     
     If $e \notin T$, then $T \cup \{ e\}$ contains a unique cycle. The \textit{fundamental cycle} of the external edge $e$ is the set of all edges belonging to this unique cycle, and is denoted by $\opcyc(e, T)$. 
\end{definition}

\begin{definition}
	 Let $G$ be a connected graph with a linear ordering on its set of edges. Let $T$ be a spanning tree in $G$. An internal edge $e_i \in T$ is said to be \textit{internally active} with respect to $T$ if it is the lowest-ordered edge among all edges in $\opcoc(e_i, T)$. Otherwise, $e_i$ is called \textit{internally inactive}. 
	 
	 Similarly, an external edge $e_j \notin T$ is said to be \textit{externally active} with respect to $T$ if it is the lowest-ordered edge among all edges in $\opcyc(e_j, T)$. Otherwise, $e_j$ is called \textit{externally inactive}.
\end{definition}

\begin{definition}
	Let $L$ be a link universe. Perform a checkerboard coloring of the planar regions determined by $L$, such that adjacent regions receive different colors and the unbounded region is colored white.
	
	The \textit{checkerboard graph} $G_{L}$ is defined as follows: place a vertex in each black region. Connect two vertices by an edge $e_c$ if and only if their corresponding black regions share a crossing point $c$ of $L$.
	
	The \textit{dual checkerboard graph} $\Gbar_{L}$ is the planar dual of $G_{L}$: its vertices correspond to the faces of $G_{L}$, and an edge $\ebar_c$ connects two vertices precisely when their corresponding faces share the edge $e_c$ in $G_{L}$.
\end{definition}

\begin{definition}
	Let $L$ be a link universe with $n$ crossing points. An \textit{ordered labeling} $l$ is a labeling of the $n$ crossing points with numbers from $1$ to $n$. Each ordered labeling of the crossing points induces a corresponding linear order on the edges of $G_L$ and $\Gbar_L$. We denote the edge $e_c$ (resp. $\ebar_c$) by $e_i$ (resp. $\ebar_i$) if the crossing point $c$ is assigned the label $i$.
\end{definition}

The following lemma is a classical result in graph theory; a detailed discussion can be found in Section 4.5 of \cite{gross2005graph}.

\begin{lemma}\label{lem:cyccocycdual}
	Let $L_l$ be a link universe with an ordered labeling $l$. Let $G_L$ and $\Gbar_L$ denote the checkerboard graph and its dual, respectively.
	\begin{itemize}
		\item [(i)] Let $T$ be a spanning tree of $G_L$. Then the set of edges $\ebar_i$ such that $e_i \notin T$ forms a spanning tree in $\Gbar_L$. This tree is called the \textit{dual spanning tree} of $T$ and is denoted by $\Tbar$.
		\item [(ii)] Let $T$ be a spanning tree of $G_L$, and suppose $e_i \in T$ and $e_j \notin T$. Then
		\begin{itemize}
			\item [(a)] $e_j \in \opcoc(e_i, T)$ if and only if $\ebar_j \in \opcyc(\ebar_i, \Tbar)$.
			\item [(b)] $e_i \in \opcyc(e_j, T)$ if and only if $\ebar_i \in \opcoc(\ebar_j, \Tbar)$.
		\end{itemize} 
	\end{itemize}
\end{lemma}

\begin{proposition}\label{prop:dualedgeactletter}
	Let $L_l$ be a link universe with an ordered labeling $l$. Let $G_L$ and $\Gbar_L$ denote the checkerboard graph and its dual, respectively. Let $T \subset G_L$ be a spanning tree, and $\Tbar \subset \Gbar_L$ its dual spanning tree. Then
	\begin{itemize}
	\item [(a)] $e_i$ is internally (resp. externally) active with respect to $T$ if and only if $\ebar_i$ is externally (resp. internally) active with respect to $\Tbar$.
	\item [(b)] $e_i$ is internally (resp. externally) inactive with respect to $T$ if and only if $\ebar_i$ is externally (resp. internally) inactive with respect to $\Tbar$.
\end{itemize} 
\end{proposition}
\begin{proof}
If $e_i \in T$, then by Lemma~\ref{lem:cyccocycdual}, the edges in the fundamental cycle $\opcyc(\ebar_i, \Tbar)$ are exactly the duals of the edges in the fundamental cocycle $\opcoc(e_i, T)$. It follows that $e_i$ is the lowest-ordered edge in $\opcoc(e_i, T)$ if and only if its dual $\ebar_i$ is the lowest-ordered in $\opcyc(\ebar_i, \Tbar)$. The case for $e_i \notin T$ follows similarly. This completes the proof.
\end{proof}

\begin{definition}
	Let $\Ltil$ be a link diagram, we associate to each crossing point $c$ a sign $\opsign(c) \in \{+, -\}$, determined by both the over-under information at $c$ and the checkerboard coloring, as illustrated in Figure~\ref{fig:trefoilchecker}. 
	
	The \textit{signed checkerboard graph} $\widetilde G_{\Ltil}$ is obtained from the unsigned checkerboard graph $G_{L}$ by assigning to each edge $e_c$ a sign $\opsign(e_c) \in \{+, -\}$, defined by $\opsign(e_c) \coloneqq \opsign(c)$. 
	
	Similarly, the \textit{signed dual checkerboard graph} $\widetilde {\Gbar}_{\Ltil}$ is obtained by assigning to each edge $\ebar_c$ the opposite sign $\opsign(\ebar_c) \coloneqq -\opsign(c)$.
\end{definition}

\begin{figure}[htp]
	\centering
	\includegraphics[width=1\textwidth]{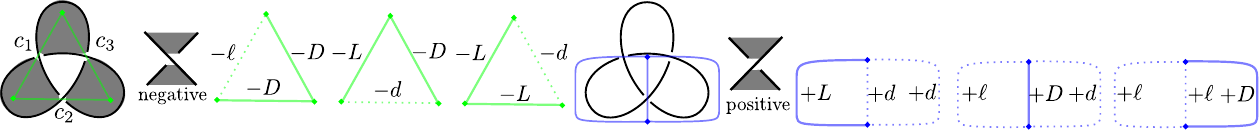}
	\caption{The signed checkerboard graph (in green) and its dual graph (in blue) associated with the trefoil, shown together with spanning trees. Each edge is labeled with its corresponding signed activity letter.}
	\label{fig:trefoilchecker}
\end{figure}

\begin{definition}
Let $\Gtil$ be a signed graph, and let $G$ be the underlying graph without sign. Let $T$ be a spanning tree of $G$. For each edge $e \in G$, we assign an \textbf{activity letter} relative to $T$:
\begin{itemize}
	\item [\textbf{--}] $L$ if $e$ is internally active, and $D$ if internally inactive. 
	\item [\textbf{--}] $\ell$ if $e$ is externally active, and $d$ if externally inactive.
\end{itemize} 
We denote this letter by $\alpha_T(e)$. The \textbf{signed activity letter} is then defined as 
\[\altil_T(e) \coloneqq \opsign(e)\alpha_T(e),\]
where $\opsign(e) \in \{+,-\}$ is the sign of the edge $e$ in $\widetilde G$.
\end{definition}

\begin{table}[h!]
	\centering
	\renewcommand{\arraystretch}{1.2}
	\begin{tabularx}{\textwidth}{@{}>{\raggedright\arraybackslash}p{6cm}*{8}{>{\centering\arraybackslash}X}@{}}
		\toprule
	signed activity letters $\altil$ &	$+L$ & $+D$ & $+\ell$ & $+d$ & $-L$ & $-D$ & $-\ell$ & $-d$ \\
		\midrule
	 evaluation for bracket polynomial $\altil|_K$ 
		& $-A^{-3}$ & $A$ & $-A^3$ & $A^{-1}$ & $-A^3$ & $A^{-1}$ & $-A^{-3}$ & $A$ \\
		\bottomrule
	\end{tabularx}
	\caption{Bracket polynomial evaluations for signed activity letters.}
	\label{tab:braeva}
\end{table}

\begin{theorem}\cite[Thistlethwaite]{thistlethwaite1987spanning}
Let $\Ltil_l$ be a link diagram with an ordered labeling $l$, and let $\widetilde G_{\Ltil}$ be its associated signed checkerboard graph. Then the Kauffman bracket polynomial admits a spanning tree expansion using the evaluations listed in Table~\ref{tab:braeva}:
	\begin{equation}
		\Gamma_{\widetilde G_{\Ltil}}(A)=\sum_{T \subset \widetilde G_{\Ltil} }\prod_{e \in \widetilde G_{\Ltil}} \altil_T(e)\big|_K,
		\label{eq:bracketpolyspaneva}
	\end{equation}
where the sum is taken over all spanning trees $T$ of $\widetilde G_{\Ltil}$, and $\altil_T(e)\big|_K$ denotes the evaluation of the signed activity letter according to Table~\ref{tab:braeva}.
\end{theorem}

\begin{remark}
Note that although the activity letter of each edge depends on a specific choice of ordered labeling, the link invariant itself is independent of such a choice. This fact is crucial for our further discussion.
\end{remark}

\section{DOUBLE SPANNING TREES AND PERFECT MATCHINGS}

In this section, we introduce the \textit{EI-property}, which governs the compatibility between the spanning tree model and the Kauffman state model. We adopt the \textit{perfect matching formulation} of the Kauffman state model, which facilitates visualization of the correspondence between the two models (see Figure~\ref{fig:PTbijection}). We then turn to the \textit{Clock Theorem}, and use it to define the \textit{states lattice polynomial}, which will play a key role in the proof of the second main theorem.

\subsection{The evaluation identity property}

This subsection is primarily inspired by \cite{cohen2012determinant}. We begin by reformulating the bracket polynomial as a sum over spanning trees and their duals, which allows us to restrict our attention to internal edges relative to spanning trees.

\begin{definition}
The	\textit{double checkerboard graph} $G^d_L$ associated with a link universe $L$ is defined as the disjoint union of the checkerboard graph $G_L$ and its dual $\Gbar_L$, that is, $G^d_L \coloneqq G \sqcup \Gbar$. A \textit{double spanning tree} $\That$ of $G^d_L$ is the disjoint union of a spanning tree $T \subset G$ and its dual spanning tree $\Tbar \subset \Gbar$, that is $\That \coloneqq T \sqcup \Tbar$. A \textit{double rooted spanning tree} $\That(v, \overline v)$ is a double spanning tree $\That$ with two distinguished vertices $v^B \in v(G)$ and $v^W \in v(\Gbar)$.
\end{definition}

\begin{remark}
For simplicity of notation, we denote by $\ehat$ the edges in the double checkerboard graph $G^d_L$. In particular, symbols like $\ehat_c$ or $\ehat_i$ may refer to either $e_c$ (or $e_i$) or their duals $\ebar_c$ (or $\ebar_i$).
\end{remark}

\begin{proposition}
Let $\Ltil_l$ be a link diagram with an ordered labeling $l$, and let $\Gtil^d _{\Ltil}$ denote its associated signed double checkerboard graph. Then the Kauffman bracket polynomial of $\Ltil$ admits an expansion over double spanning trees of $\Gtil^d _{\Ltil}$. More precisely, we have
	\begin{equation}
		\Gamma_{\Gtil^d_{\Ltil}}(A)=\sum_{\That = T \sqcup \overline{T} }
		\left( \prod_{e \in T} \altil_T(e)\big|_K \right)
		\left( \prod_{\ebar \in \Tbar} \altil_{\Tbar}(\ebar)\big|_K \right),
		\label{eq:bracketpolydoublespaneva}
	\end{equation}
	summing over all double spanning trees $\That$ of $G^d_{L}$.
\end{proposition}
\begin{proof}
Let $e_i \in G$ be an edge with a sign, and suppose it is externally active (resp. inactive) with respect to a spanning tree $T$. Then by Proposition~\ref{prop:dualedgeactletter}, its dual edge $\ebar_i$ is internally active (resp. inactive) with respect to $\Tbar$. Moreover, the signs satisfy $\opsign(\ebar_i)$ = $-\opsign(e_i)$. Note that the corresponding signed activity letters yield the same evaluation in Table~\ref{tab:braeva}; that is, $\altil_T(e_i)|_K =\altil_{\Tbar}(\ebar_i)|_K$. Hence, the formula follows directly from Equation~\eqref{eq:bracketpolyspaneva}.
\end{proof}

The following three definitions define the perfect matching version of the Kauffman state model.

\begin{definition}
Let $L$ be a link universe, the \textit{overlaid checkerboard graph} $G^{ov}_L$ associated with $L$ is a plane bipartite graph defined as follows:\\
\textbf{--} Place a \textit{crossing vertex} $v_c$ at each crossing point $c$. \\
\textbf{--} Place a \textit{round vertex} $v_R$ in each region $R$ of $L$. \\
\textbf{--} Draw a \textit{half-edge} $h_{(c, R)}$ connecting $v_c$ and $v_R$ if and only if the crossing point $c$ is incident to the region $R$.
\end{definition}

\begin{definition}
Let $L_s$ be a link universe with a distinguished segment $s$, the two regions adjacent to $s$ has different colors; denote the corresponding round vertices by $v^{W}_s$ and $v^{B}_s$, representing the white and black regions, respectively. The \textit{balanced overlaid checkerboard graph} $G^b_{L,s}$ is obtained from the overlaid cheakerboard graph $G^{ov}_L$ by deleting $v^{W}_s$ and $v^{B}_s$, together with all the half-edges incident to these vertices.
\end{definition}

\begin{figure}[htp]
\centering
\includegraphics[width=0.7\textwidth]{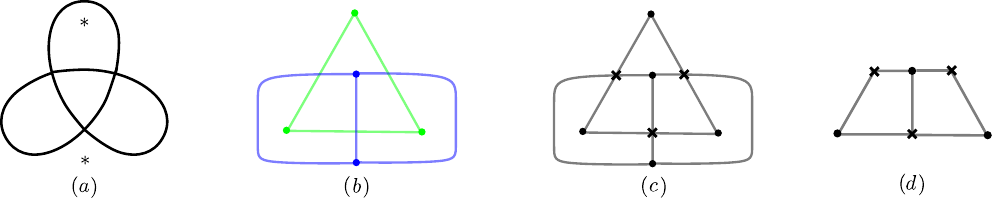}
\caption{(a) The trefoil universe with the two regions adjacent to the distinguished segment marked with stars. (b) The double checkerboard graph associated to the universe. (c) The overlaid checkerboard graph. (d) The balanced overlaid checkerboard graph.}
\label{fig:4graphs}
\end{figure}

\begin{remark}
Note that the vertices of the double checkerboard graph $G^d_L$ coincide with the round vertices of the overlaid checkerboard graph $G^{ov}_L$. Each full-edge $\ehat_c \in G^d_L$ corresponds to a pair of half-edges $h_{(c, R)} \in G^{ov}_L$ that are incident to the crossing vertex $v_c$. Additionally, each square face of $G^{ov}_L$ corresponds to a segment of $L$; see Figure~\ref{fig:4graphs} for an illustration.
\end{remark}

\begin{definition}
A \textit{perfect matching} of a graph $G$ is a subset of the edges in which each  vertex of $G$ is incident to exactly one edge.
\end{definition}

\begin{definition}
For each rooted spanning tree $T(v) \subset G$, assign a direction to each edge $e \in T$ so that every vertex in $G$ can be reached from the root $v$ by an outwardly oriented path in $T$. For each edge $e$, the vertex to which $e$ is directed is called the \textit{target} of $e$ in $T(v)$.
\end{definition}

\vspace{0.5em}
The correspondence characterized in the following proposition connects the perfect matching formulation of the Kauffman state model with the spanning tree model.

\begin{proposition}
Let $L_s$ be a link universe with a distinguished segment $s$, and let $G^d_L$ and $G^b_{L,s}$ denote the associated double checkerboard graph and the balanced overlaid checkerboard graph, respectively. Then there is a bijection between the perfect matchings $\Pcal \subset G^b_{L,s}$ and the double rooted spanning trees $\That\big(v^{W}_s, v^{B}_s\big)\subset G^d_L$. More precisely, the mutually inverse bijections are given by
\[
\eta: \Pcal \to \That \big(v_s^W, v_s^B\big), \quad 
h \mapsto \ehat \quad \text{where } \ehat \text{ is the full-edge corresponding to } h,\]
\[
\theta: \That\big(v_s^W, v_s^B\big) \to \Pcal, \quad 
\ehat_c \mapsto h_{(c, R)} \quad \text{where } v_R \text{ is the target of $\ehat_c$ in $\That\big(v_s^W, v_s^B\big)$}.\]
\label{prop:mathingtreecor}
\end{proposition}
\begin{proof}
We first show that $\theta\big(\That(v_s^W, v_s^B)\big)$ is indeed a perfect matching. Since $\That$ contains spanning subgraphs of both $G$ and its dual, every crossing vertex and every round vertex in $G^b_{L,s}$ is incident to at least one half-edge in $\theta \big(\That(v_s^W, v_s^B)\big)$. By construction, each crossing vertex is incident to exactly one such half-edge. Moreover, since both spanning trees are acyclic, each round vertex is also incident to exactly one half-edge. It follows that $\theta\big(\That(v_s^W, v_s^B)\big)$ is a perfect matching.

To see $\eta(\Pcal)$ is a double spanning tree, we first claim that it contains no cycle. Suppose for contradiction that $\eta(\Pcal)$ does contain a cycle. Then none of the deleted vertices can lie on this cycle, since this would imply the existence of a round vertex in $\Pcal$ incident to two edges, contradicting the perfect matching condition. Note that $\eta(\Pcal)$ consists of two disjoint parts: $\eta(\Pcal) \cap G_L$ and $\eta(\Pcal) \cap \Gbar_L$. Moreover, for each crossing $c$, exactly one of the edges $e_c$ and $\ebar_c$ belongs to $\eta(\Pcal)$. Without loss of generality, suppose there exists a cycle $C_1$ in $\eta(\Pcal) \cap G_L$ that does not contain $v_s^B$, and assume that $v_s^W$ lies inside $C_1$. The cycle $C_1$ separates $\eta(\Pcal) \cap \Gbar_L$ into two disconnected components. By construction, the component outside $C_1$ contains equal numbers of vertices and edges, and hence must contain a cycle, denoted by $\overline C_2$. Note that the two omitted vertices $v_s^W$ and $v_s^B$ correspond to regions adjacent to the segment $s$. In particular, $v_s^B$ lies on the cycle in $G_L$ that surrounds the face corresponding to $v_s^W$. Therefore, $v_s^B$ also lies inside the cycle $C_1$. It follows that the component of $\eta(\Pcal) \cap G_L$ outside the cycle $\overline C_2$ contains a cycle $C_3$. Repeating this procedure infinitely leads to a contradiction, as the graph is finite. This proves the claim. Finally, if one of $\eta(\Pcal) \cap G_L$ and $\eta(\Pcal) \cap \Gbar_L$ is not spanning, then the other must contain a cycle, contradicting the previous claim. Therefore, $\eta(\Pcal)$ is indeed a double spanning tree of $G_L^d$.

Now, we show that $\eta$ and $\theta$ are mutually inverse bijections. The identity $\eta \circ \theta (\That)=\That$ is immediate. For each perfect matching $\Pcal \subset G^b_{L,s}$, we have already shown that $\eta (\Pcal)$ is a double spanning tree. To see that $\theta \circ \eta (\Pcal)=\Pcal$, we need to verify that for each half-edge $h_{(c, R)}$, the vertex $v_R$ is the target of $\eta(h)$ in $\eta (\Pcal)$. Without loss of generality, assume $R$ is a black reigion. We proceed by induction on the length of the outwardly oriented path from the root $v^B_s$ in which $\eta(h)$ is the final edge. If the path length is 1, that is, $v^B_s$ is an endpoint of the edge $\eta(h)$, the claim is clear. Now, assume the claim holds for all paths of length less than $m$. Consider a path of length $m$ in which $\eta\big(h_{(c,R)}\big)$ is the final edge, if $v_R$ is not the target of $\eta(h)$, it must be a target of another edge, which would be the final edge in a path of length $m-1$. By the inductive hypothesis, this would imply that $v_R$ is incident to two edges in $\Pcal$, contradicting the perfect matching condition. Hence, $v_R$ is indeed the target of $\eta(h)$, as required. This completes the proof that $\theta \circ \eta(\Pcal)=\Pcal$.
\end{proof}

\vspace{0.5em}
After establishing the correspondence between the two models, we now define the \textit{EI-property}, which ensures that the activity letters assigned in the spanning tree model can be transferred to the perfect matchings.

\begin{definition}
	A \textbf{link configuration} is a link universe $L_{s;l}$ equipped with a distinguished segment $s$ and an ordered labeling $l$. Let $G^b_{L, s}$ denote the associated balanced overlaid checkerboard graph. 
	
	For a fixed half-edge $h \in G^b_{L, s}$, let $\mathcal{PM}(h)$ denote the set of perfect matchings of $G^b_{L, s}$ that contain $h$. Recall that $\eta(\Pcal)$ is the double spanning tree associated with the perfect matching $\Pcal$, and $\alpha_{\eta(\Pcal)}\big(\eta(h)\big)$ is the activity letter assigned to the full-edge $\eta(h)$ with respect to the double spanning tree $\eta(\mathcal{P})$.
	
	We say that $h$ satisfies the \textbf{evaluation identity property} (abbreviated as \textbf{EI-property}) if 
	\[\alpha_{\eta (\Pcal)}\big(\eta (h)\big) =\alpha_{\eta (\Pcal')}\big(\eta (h)\big) \quad \text{for all } \mathcal{P}, \mathcal{P}' \in \mathcal{PM}(h).\]
	
	 We say that a link configuration $L_{s; l}$ satisfies the $\textit {EI-property}$ if every half-edge in $G^b_{L, s}$ satisfies the $\textit {EI-property}$.
\end{definition}

We present an example and a counterexample here to illustrate the \textit{EI-property}.

\begin{example}
 Let $L_s$ be the trefoil universe with a distinguished segment $s$. Let $G^b_{L,s}$ and $G^d_{L}$ denote its balanced overlaid checkerboard graph and double checkerboard graph, respectively. The bijection of the perfect matchings of $G^b_{L,s}$ and the double rooted spanning trees of $G^d_{L}$ is depicted in subfigure (a) of Figure~\ref{fig:PTbijection}. Subfigure (b) shows two different ordered labelings $l$ and $l'$. 
 
 The link configuration $L_{s;l}$ satisfies the \textit{EI-property}, while $L_{s;l'}$ does not. As illustrated in the figure, in the configuration $L_{s;l'}$, the activity letters assigned to the full-edge $\eta (h)$ satisfy:
 \[\alpha_{\eta (\Pcal_1)}\big(\eta (h)\big) =D \quad \text{ and } \quad \alpha_{\eta (\Pcal_2)}\big(\eta (h)\big)=L,\]
 which shows that the half-edge $h$ does not satisfy the \textit{EI-property}.
\end{example}

\begin{figure}[hbp]
	\centering
	\subfigure[]{
		\includegraphics[width=0.45\textwidth, height=0.15\textheight]{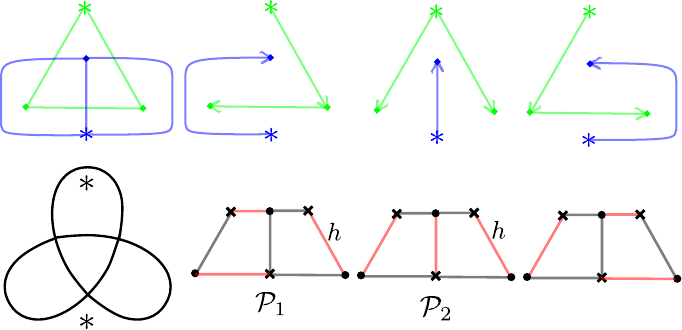}}
	\hspace{2mm}
	\subfigure[]{
		\includegraphics[width=0.45\textwidth, height=0.15\textheight]{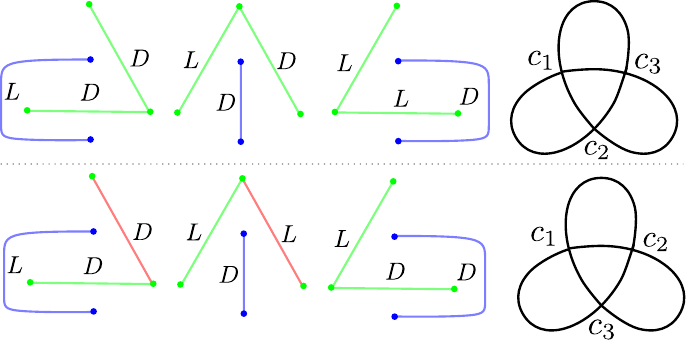}}
	\caption{(a) The bijection between the perfect matchings (shown in red) of $G^b_{L,s}$ and the double rooted spanning trees of $G^d_{L}$ with the two roots indicated by stars. (b) Two different ordered labelings $l$ (top) and $l'$ (bottom).}
	\label{fig:PTbijection}
\end{figure}

\begin{definition}
	If a link configuration $L_{s; l}$ has the $\textit {EI-property}$, we define the activity letter $\alpha (h)$ for each half-edge $h$ in $G^b_{L, s}$ by 
	\[\alpha (h) \coloneqq \alpha_{\eta (\Pcal)}\big(\eta (h)\big), \]
	where $\Pcal \in \mathcal{PM}(h)$ is any perfect matching containing $h$.
\end{definition}

\begin{definition}
Assign each region a sign by setting $\opsign(R)=+$ if $R$ is a black region and $\opsign(R)=-$ if R is a white region. The sign of a half-edge $h_{(c,R)}$ is defined as 
\[\opsign\big(h_{(c,R)}\big) \coloneqq \opsign(c)\opsign(R).\] 
If $L_{s; l}$ has the $\textit {EI-property}$, the \textit{signed activity letter} of a half-edge $h_{(c,R)}$ is defined as $\altil (h) \coloneqq \opsign(h) \alpha (h)$.
\end{definition}

\begin{proposition}
	Let $\Ltil$ be a link diagram, and let $L$ be its corresponding link universe. If there exists a link configuration $L_{s; l}$ satisfying the $\textit {EI-property}$, then the Kauffman bracket polynomial of $\Ltil$ admits a perfect matching expansion  
	\begin{equation}
	\Gamma_{G^b_{L, s}}(A)=\sum_{\Pcal}\prod_{h \in \Pcal} \altil (h)\big|_K,
	\end{equation}
	where the sum is taken over all perfect matchings of the balanced overlaid checkerboard graph $G^b_{L, s}$.
\end{proposition}
\begin{proof}
 Fix a perfect matching $\Pcal \subset G^b_{L, s}$, and let $h$ be a half-edge in $\Pcal$. It is clear that $\opsign(h)=\opsign\big(\eta(h)\big)$. Moreover, since the link configuration $L_{s; l}$ satisfies the $\textit {EI-property}$, we have $\alpha (h)=\alpha_{\eta(\Pcal)} (\eta(h))$. It then follows that $\altil (h)=\altil_{\eta(\Pcal)} (\eta(h))$. By the bijection established in Proposition~\ref{prop:mathingtreecor}, the formula follows immediately from Equation~\eqref{eq:bracketpolydoublespaneva}. 
\end{proof}

\begin{proposition}\label{prop:HopfEI}
Let $H_{s;l_H}$ denote the Hopf link universe $H$, equipped with a distinguished segment $s$ and an arbitrary ordered labeling $l_H$ of its crossing points. Then $H_{s; l_H}$ always satisfies the $\textit {EI-property}$.
\end{proposition}
\begin{proof}
Note that for any segment $s$ in the Hopf link universe, in the corresponding balanced overlaid checkerboard graph $G^b_{L, s}$, each half-edge $h$ appears only in exactly one perfect matching of $G^b_{L, s}$.
\end{proof}

\subsection{The lattice of perfect matchings}

This subsection is primarily inspired by \cite{meszaros2024dimer}. We begin by recalling the definitions of the \textit{Kauffman states} and the \textit{Clock Theorem}, and demonstrate that perfect matchings and Kauffman states reflect the same combinatorial structure.

\begin{definition}
Let $L_s$ be a link universe with a distinguished segment $s$. A \textit{marker} is a pair $(c,R)$, where $c$ is a crossing point and $R$ is a region incident to $c$. A \textit{Kauffman state} relative to $s$ is a set of markers such that each crossing point of $L$, and each region of $L$ except for the two adjacent to $s$, appears in exactly one marker.
\end{definition}

\begin{definition}
A state $\Scal'$ is said to be obtained from $\Scal$ by a \textit{counterclockwise transposition} at a segment $s_i$ if $\Scal'$ is obtained by switching the corner markers adjacent to $s_i$ in a counterclockwise manner, as illustrated in subfigure (a) of Figure \ref{fig:upmove}.
\end{definition}

\begin{figure}[hbp]
	\centering
	\subfigure[]{
		\includegraphics[width=0.3\textwidth, height=0.07\textheight]{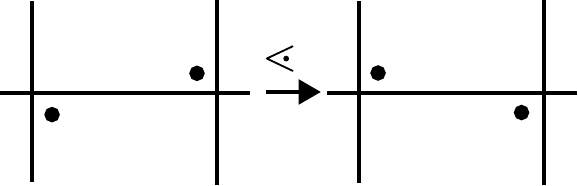}}
	\hspace{10mm}
	\subfigure[]{
		\includegraphics[width=0.3\textwidth, height=0.07\textheight]{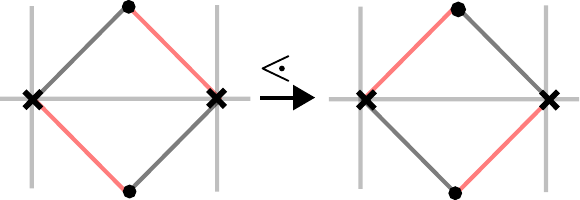}}
	\caption{(a) The counterclockwise transposition. (b) The up move.}
	\label{fig:upmove}
\end{figure}

\begin{theorem}\cite[Clock Theorem]{kauffman1983formal}\label{thm:clockthm}
Define a partial order on the set of Kauffman states by $\Scal_1 \leq \Scal_2$ if $\Scal_2$ can be obtained from $\Scal_1$ by a sequence of counterclockwise transpositions. Then the set of Kauffman states, equipped with this partial order, forms a bounded lattice. Denote the minimal and maximal states in this lattice by $\Scal_{\min}$ and $\Scal_{\max}$, respectively.
\end{theorem}

\begin{definition}
Let $L_s$ be a link universe with a distinguished segment $s$, and let $G^b_{L,s}$ denote its associated balanced overlaid checkerboard graph. 

A segment $s_j$ is called a \textbf{transposable segment} if $s_j$ and $s$ do not bound the same region.
 
A square faces $f_j$ in $G^b_{L,s}$ is called a \textbf{movable face} if it corresponds to a transposable segment $s_j$. 

A perfect matching $\Pcal' \subset G^b_{L,s}$ is said to be obtained from a matching $\Pcal$ by an \textit{up move} at the square face $f_j$ if $\Pcal'$ is obtained by replacing the two matching edges of $f_j$ that run clockwise from a round vertex to a crossing vertex with the two complementary edges that run from a crossing vertex to a round vertex, as illustrated in subfigure (b) of Figure~\ref{fig:upmove}. 
\end{definition}

\begin{proposition}\label{prop:statematchingcor}
Let $L_s$ be a link universe with a distinguished segment $s$. Then there exists a natural bijection between the set of perfect matchings $\Pcal \subset G^b_{L,s}$ and the set of Kauffman states $\Scal$ relative to $s$. More precisely, this bijection is given by the correspondence:
\[\Pcal \leftrightarrow \Scal, \ h_{(c, R)} \leftrightarrow (c,R),\]
where each matched half-edge $h_{(c, R)}$ in $\Pcal$ corresponds to a marker $(c, R)$ in the state $\Scal$.
\end{proposition}
\begin{proof}
In $G^b_{L,s}$ each perfect matching selects exactly one incident region for each crossing point, and vice versa, matching the definition of a Kauffman state relative to $s$.
\end{proof}

\begin{corollary}
Define a partial order on the set of perfect matchings by $\Pcal_1 \leq \Pcal_2$ if $\Pcal_2$ can be obtained from $\Pcal_1$ by a sequence of up moves. Then the set of perfect matchings, equipped with this partial order, forms a bounded lattice.
\end{corollary}
\begin{proof}
It is evident from Figure~\ref{fig:upmove} that an up move at the square face $f_j$ in a perfect matching corresponds precisely to a counterclockwise transposition at the segment $s_j$ in the associated Kauffman state. Therefore, the result follows directly from Theorem~\ref{thm:clockthm} and Proposition~\ref{prop:statematchingcor}.
\end{proof}

For an example of the lattice bijection between of Kauffman states and perfect matchings, we refer to Figure~\ref{fig:stateslattice}.

\begin{figure}[hbp]
	\centering
	\subfigure[]{
		\includegraphics[width=0.47\textwidth, height=0.38\textheight]{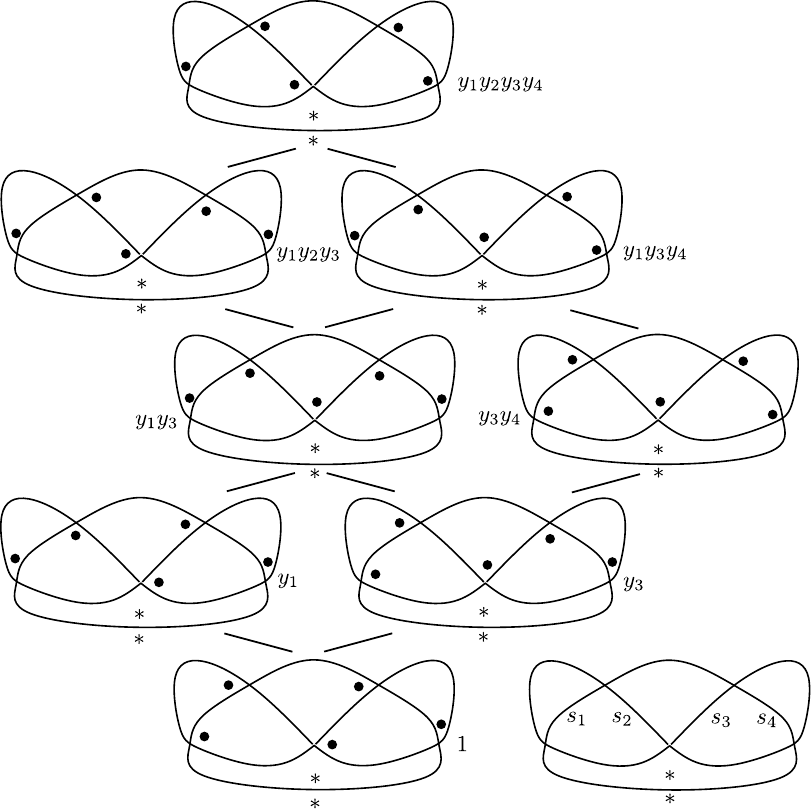}}
	\hspace{3mm}
	\subfigure[]{
		\includegraphics[width=0.47\textwidth, height=0.38\textheight]{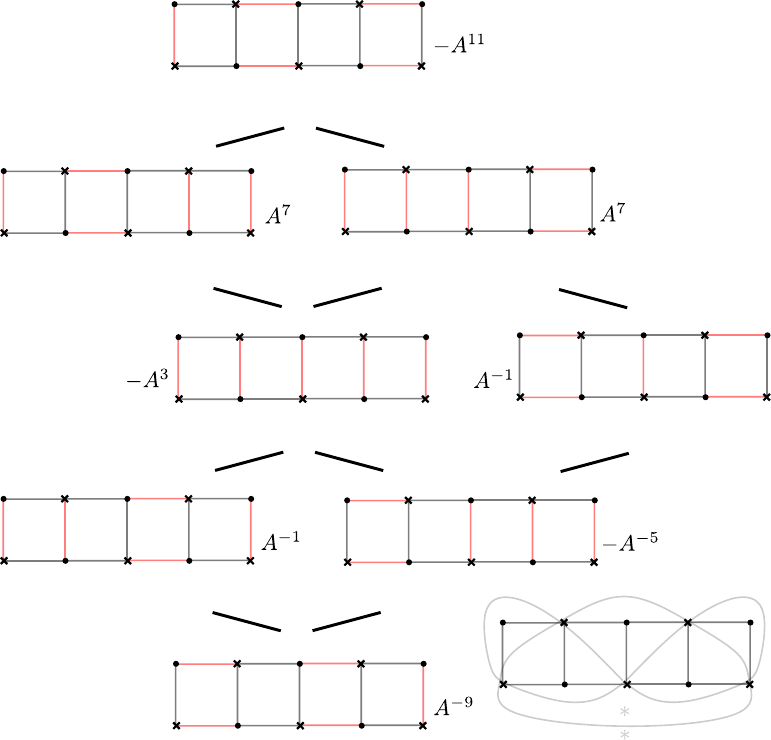}}
	\caption{(a) The lattice of Kauffman states associated with the Whitehead link universe, with each state labeled by its corresponding state monomial. (b) The corresponding lattice of perfect matchings reflecting the same combinatorial structure.}
	\label{fig:stateslattice}
\end{figure}

We use the established correspondence between perfect matchings and Kauffman states to introduce several notions needed for the proof of the second main theorem in Section~5.

\begin{definition}
Let $L_{s; l}$ be a link configuration satisfying the $\textit {EI-property}$, and let $\Pcal_{\Scal}$ denote the perfect matching associated to the Kauffman state $S$. The \textbf{weight} $\om(\Scal)$ of the state $\Scal$ is defined by 
\[\om(\Scal) \coloneqq \prod_{h \in \Pcal_{\Scal}}\altil (h)|_K.\]
The \textbf{weight ratio} of each transposable segment $s_i$ is then defined as \[\om(s_j) \coloneqq \om(\Scal')/\om(\Scal),\]
where $\Scal'$ is the state obtained from $\Scal$ by a counterclockwise transposition at $s_j$.
\end{definition}

\begin{definition}
Let $L_s$ be a link universe with a distinguished segment $s$. Let $\Scal$ be a Kauffman state, a \textit{saturated chain} $C_{\Scal}$ of $\Scal$ is a sequence of counterclockwise transpositions from the minimal state $\Scal_{\min}$ to $\Scal$: \[C_{\Scal} \coloneqq \Scal_{\min}=\Scal_0 \lessdot \Scal_1 \cdots \lessdot \Scal_l=\Scal.\]
Suppose that $\Scal_i$ is obtained from $\Scal_{i-1}$ by a counterclockwise transposition at segment $s^i$. The associated \textit{transposition words} of $C_{\Scal}$ is the words $s^1\ldots s^l$. The \textit{state monomial} of a state $\Scal$ with respect to $C_{\Scal}$ is defined as $m_C(\Scal) \coloneqq y_{s^1}\ldots y_{s^l}$.
\end{definition}

\begin{proposition}
Let $L_s$ be a link universe with a distinguished segment $s$. Let $\Scal$ be a Kauffman state, and let $C_{\Scal}$ and $C'_{\Scal}$ be any two saturated chains from $\Scal_{\min}$ to $\Scal$. Then the corresponding state monomials are equal: $m_C(\Scal)=m_{C'}(\Scal)$. Hence, the \textbf{state monomial} of 
$\Scal$ is well-defined, independent of the choice of saturated chain, and we may denote it simply by $m(\Scal)$.
\end{proposition}
\begin{proof}
It was shown in \cite{gilmer1986duality} that the lattice of kauffman states defined above is indeed a graded distributive lattice. Therefore, the chains $C_{\Scal}$ and $C'_{\Scal}$ have the same length, and we can proceed by induction on the length of such saturated chains. The base case, when the saturated chains have length 1, is straightforward. Now, assume that the desired statement holds for all Kauffman states whose associated saturated chains have length less than $n$. Denote by $S_{n-1}$ and $S'_{n-1}$ the $(n-1)st$ Kauffman state in the saturated chains $C_{\Scal}$ and $C'_{\Scal}$, respectively. Suppose $S_{n-1}$ and $S'_{n-1}$ can be obtained from $\Scal$ by a clockwise transposition at segments $s^n$ and $s'^n$, respectively. Define a new Kauffman state $\Scal''$ as the result of applying a clockwise transposition at segment $s^n$ to $\Scal'_{n-1}$. Then, by the induction hypothesis, we have 
\[m(\Scal_{n-1})=m(\Scal'')\cdot y_{s'^n} \quad \text{ and } \quad m(\Scal'_{n-1})=m(\Scal'')\cdot y_{s^n}.\]
On the other hand, the state monomials associated with the chains $C_{\Scal}$ and $C'_{\Scal}$ are given by
\[m_{C}(\Scal)=m(\Scal_{n-1})\cdot y_{s^n} \quad \text{ and } \quad m_{C'}(\Scal)=m(\Scal'_{n-1})\cdot y_{s'^n}.\]
Substituting the previous expressions yields $m_{C}(\Scal)=m_{C'}(\Scal)$, which completes the inductive step.
\end{proof}
	
\begin{definition}
Let $L_s$ be a link universe with a distinguished segment $s$. The \textbf{states lattice polynomial} of $L_s$ is defined as \[\Mcal_{L,s}(\ybf)= \sum_{\Scal}m(\Scal),\] where the sum is taken over all Kauffman states of $L$ relative to $s$, and $m(\Scal)$ denotes the state monomial associated to the state $\Scal$.
\end{definition}

\begin{remark}
For simplicity of notation, we denote by $y_j$ the variable corresponding to the transposable segment $s_j$. Let $L_s$ be a Whitehead link universe with a distinguished segment $s$. The state monomial $m(\Scal)$ associated with each Kauffman state $\Scal$ is illustrated in subfigure (a) of Figure~\ref{fig:stateslattice}. Similar polynomials also appear in \cite{bazier2022knot} and \cite{meszaros2024dimer}, though defined with slight variations. 
\end{remark}

\section{ADMISSIBLE BIGON EXTENSION}

In this section, we introduce an operation called the \textit{admissible bigon extension} on a link universe $L$, which produces a new \textit{link configuration} satisfying the \textit{EI-property} from an existing one. We provide a criterion for identifying when a bigon extension is not admissible and use it to prove the first main theorem. We then present a method, referred to as the \textit{ASI procedure}, for constructing link configurations globally. We show that all \textit{admissible link configurations} can be completely characterized by this construction and use it to establish the inclusion of classical links.

\subsection{Criteria for admissible bigon extensions}

We begin by defining the two types of bigon extensions. Without loss of generality, we focus in this subsection primarily on bigon extensions of type A.

\begin{figure}[htp]
	\centering
	\subfigure[]{
		\includegraphics[width=0.47\textwidth, height=0.1\textheight]{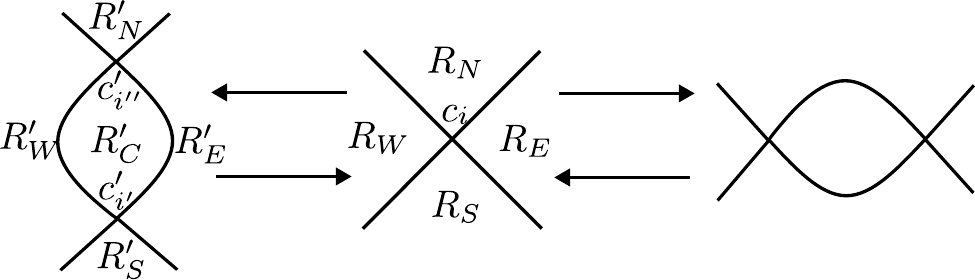}}
	\hspace{3mm}
	\subfigure[]{
		\includegraphics[width=0.47\textwidth, height=0.1\textheight]{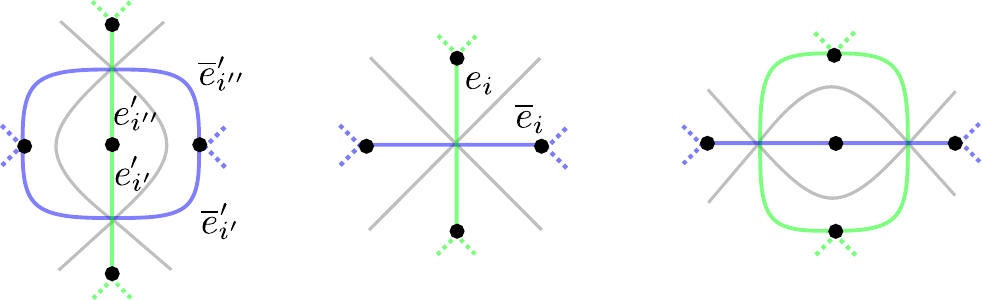}}
	\caption{(a) Two types of bigon extensions. (b) The induced local structure in the double checkerboard graph.}
	\label{fig:bigonextension}
\end{figure}
	
\begin{definition}
Let $L$ be a link universe. A \textit{bigon extension} $E_c$ at a crossing point $c$ splits $c$ into two cusps and creates a new bigon between them, as illustrated in Figure \ref{fig:bigonextension}. If the newly created region is black, the bigon extension is said to be of type A; otherwise, it is of type B. The reverse operation is called a \textit{bigon reduction}.

If a bigon extension $E_c$ is of type A, it is said to extend the edge $e_c$; otherwise, it extends the edge $\ebar_c$. 
\end{definition}

\begin{proposition}\label{prop:prime}
Let $L$ be a link universe, and let $L'$ be the link universe obtained from $L$ by a bigon extension at $c$. If $L$ is prime-like, then $L'$ is also prime-like.
\end{proposition}
\begin{proof}
Suppose $L'$ is not prime-like. Then there exists a simple closed curve $S$ that intersects $L'$ in exactly two points and separates it into two nontrivial components. If the two new crossing points lie on different sides of $S$, then the two intersect points between $S$ and $L'$ must lie on the bigon, which implies the crossing point $c$ is nugatory. If the two new crossing points lie on the same side of $S$, then applying a bigon reduction to $L'$ yields a simple closed curve that intersects $L$ in exactly two points and separates it into nontrivial parts. This implies that $L$ is not prime-like, contradicting the assumption. 
\end{proof}

\begin{definition}
Let $L_l$ be a link universe with an ordered labeling $l$, and let $L'$ be the link universe obtained from $L$ by a bigon extension $E_{c_i}$. Then there are two natural new ordered labelings of the crossing points of $L'$, defined as follows: 
\[
l'(c') =
\begin{cases}
	j, & \text{if } c' \text{ corresponds to } c_j \text{ with } j < i; \\
	i \, \text{or } i+1, & \text{if } c'=c'_{i'} \text{ or } c'_{i''};\\
	j + 1, & \text{if } c' \text{ corresponds to } c_j \text{ with } j > i.
\end{cases}
\]
\end{definition}

\begin{remark}
Let $c_j$ with $j \neq i$ be a crossing point of $L$. For simplicity of notation, we denote by $l'(j)$ the label of the crossing point of $L'$ that corresponds to $c_j$; in other words, the crossing point $c'_{l'(j)}$ corresponds to $c_j$. We also denote by $R'$ the region of $L'$ corresponding to the region $R$ of $L$, and by $\alpha_{\That}(\ehat)$ the activity letter of the edge $\ehat$ with respect to its corresponding spanning tree in $\That$.
\end{remark}

\begin{definition}
Let $L_{s; l}$ be a link configuration satisfying the \textit{EI-property}, and let $L'$ be the link universe obtained from $L$ by a bigon extension $E_{c_i}$. We say that the bigon extension $E_{c_i}$ is \textbf{admissible} if there exists a natural ordered labeling $l'$ such that the resulting configuration $L'_{s; l'}$ also satisfies the \textit{EI-property}.
\end{definition}

\vspace{0.5em}
The following theorem characterizes the admissible bigon extensions.

\begin{theorem}\label{thm:localcri}
	Let $L_{s; l}$ be a link configuration satisfying the \textit {EI-property}, and suppose the link universe $L$ is prime-like. Then the following hold: 
	\begin{itemize}
		\item [(a)] Recall that the omitted vertices in $G^d_L$ correspond to the two regions adjacent to the segment $s$. A bigon extension $E_c$ is \textbf{not} admissible if and only if it satisfies both of the following conditions:
		\begin{itemize}
			\item [(i)] Neither of the two endpoints of the edge $\ehat_c$, which is extended by the bigon extension $E_c$, is an omitted vertex.
			\item [(ii)] The two half-edges corresponding to $\ehat_c$ are assigned the same activity letter.
		\end{itemize}
		\item [(b)] If the bigon extension $E_c$ is admissible, then the order of the crossing points is uniquely determined. More precisely, if $E_c$ is of type A, then the order is given by:
		\[i' \textless i'' \quad \text{ if } \, \alpha \big(h_{(c_{i}, R_N)}\big)=L \text{ or } \alpha \big(h_{(c_{i}, R_S)}\big)=D,\]
		\[i' \textgreater i'' \quad \text{  if } \, \alpha \big(h_{(c_{i}, R_N)}\big)=D \text{ or } \alpha \big(h_{(c_{i}, R_S)}\big)=L.\] 
	\end{itemize}
\end{theorem}

Before proving Theorem~\ref{thm:localcri}, we need some preparations. The following proposition establishes the correspondence between the double spanning trees before and after a bigon extension.

\begin{proposition}\label{prop:doublespancorre}
Let $L_l$ be a link universe with an ordered labeling $l$, and let $L'_{l'}$ be the link universe obtained from $L_l$ by a type-A bigon extension $E_{c_i}$. Let $G_L^d$ and $G_{L'}^d$ denote the double checkerboard graphs associated with $L$ and $L'$, respectively. Then there is a natural correspondence between the double spanning trees $\That$ of $G_L^d$ and $\That'$ of $G_{L'}^d$, described as follows:\\
\textbf{Case 1:} If $e_i \in \That$, then $\That$ corresponds to exactly one double spanning tree: \[\That'_0=\{\ehat'_{l'(j)} \mid  \ehat_j \in \That,\, j \neq i\} \cup \{e'_{i'}\} \cup \{e'_{i''}\}.\]\\
\textbf{Case 2:} If $\ebar_i \in \That$, then $\That$ corresponds to two distinct double spanning trees: 
\[\That'_1=\{\ehat'_{l'(j)} \mid \ehat_j \in \That,\, j \neq i\} \cup \{e'_{i'}\} \cup \{\ebar'_{i''}\},\] 
\[\That'_2=\{\ehat'_{l'(j)} \mid \ehat_j \in \That,\, j \neq i\} \cup \{\ebar'_{i'}\} \cup \{e'_{i''}\}.\]
\end{proposition}
\begin{proof}
 For a visual reference, see Figure~\ref{fig:bigonextensionmatch}; disregard the orientations of the edges in the illustration.
\end{proof}

\vspace{0.5em}
We need the following lemma.

\begin{lemma}\label{lem:bigonexcontro}
	Let $L_l$ be a link universe with an ordered labeling $l$. Let $e_i$ and $e_{i+1}$ be two edges in the checkerboard graph $G_L$ that share a 2-valent vertex $v_C$. Then the following hold:
	\begin{itemize}
		\item [(a)] If $T$ is a spanning tree of $G_L$ such that $e_i \in T$ and $e_{i+1} \notin T$, then $\alpha_{T}(e_i)=L$.
		\item [(b)] If $T$ is a spanning tree of $G_L$ such that $e_{i} \notin T$ and $e_{i+1} \in T$, then $\alpha_{T}(e_{i+1})=D$.
		\item [(c)] If $\Tbar$ is a spanning tree of $\Gbar_L$ such that $\ebar_{i+1} \in \Tbar$, then $\alpha_{\Tbar}(\ebar_{i+1})=D$.
	\end{itemize}
\end{lemma}
\begin{proof} 
(a) By assumption, the single vertex $v_C$ is one of the two disconnected components of $T \setminus \{e_i\}$. Hence, the fundamental cocycle $\opcoc(e_i, T)$ of $e_i$ with respect to $T$ is $\{e_i, e_{i+1}\}$. Since $e_i$ is the lowest-ordered edge in $\opcoc(e_i, T)$, the activity letter satisfies $\alpha_{T}(e_i)=L$.\\
(b) Similar to (a), the fundamental cocycle $\opcoc(e_{i+1}, T)$ of $e_{i+1}$ with respect to $T$ is $\{e_i, e_{i+1}\}$. Since $e_{i+1}$ is not the lowest-ordered edge in $\opcoc(e_{i+1}, T)$, the activity letter satisfies $\alpha_{T}(e_{i+1})=D$.\\
(c) By assumption, $\ebar_i$ and $\ebar_{i+1}$ are two edges between the same two vertices. Hence, if $\ebar_{i+1} \in \Tbar$, then $\ebar_{i} \notin \Tbar$, and $\ebar_{i}$ is contained in the fundamental cocycle $\opcoc(\ebar_{i+1}, \Tbar)$, which implies $\ebar_{i+1}$ is not the lowest-ordered edge in the fundamental cocycle. This completes the proof.
\end{proof}

\vspace{0.5em}
The following two lemmas describe the changes in the activity letters of edges induced by a bigon extension.

\begin{lemma}	\label{lem:bigonexcorr}
Let $L_l$ be a link universe with an ordered labeling $l$, and let $L'_{l'}$ be the link universe obtained from $L_l$ by a type-A bigon extension $E_{c_i}$. Let $\That$ be a double spanning tree of $G_{L}^d$. The following hold:
\begin{itemize}
	\item [(a)] If $e_i \in \That$, let $\That'_0 \subset G_{L'}^d$ be the unique double spanning tree corresponding to $\That$. Then the activity letters $\alpha_{\That'_0}\big(e'_{i'}\big)$ and $\alpha_{\That'_0}\big(e'_{i''}\big)$ are independent of the order of $i'$ and $i''$. Moreover, we have
	\[\alpha_{\That}\big(e_{i}\big)=\alpha_{\That'_0}\big(e'_{i'}\big)=\alpha_{\That'_0}\big(e'_{i''}\big)\].
	\item [(b)] If $\ebar_i \in \That$, let $\That'_1$ and $\That'_2 $ be the two double spanning trees corresponding to $\That$ containing $\ebar_{i''}$ and $\ebar_{i'}$, respectively. Then the activity letters $\alpha_{\That'_1}\big(\ebar'_{i''}\big)$ and $\alpha_{\That'_2}\big(\ebar'_{i'}\big)$ are dependent on the order of $i'$ and $i''$. Moreover, suppose $i' \textless i''$, we have 
	\[\alpha_{\That}\big(\ebar_{i}\big)=\alpha_{\That'_2}\big(\ebar'_{i'}\big), \text{ and } \alpha_{\That'_1}\big(\ebar'_{i''}\big)=D.\]
	\end{itemize}
\end{lemma}
\begin{proof}
	(a) Let $T$ and $T'_0$ be the spanning tree parts of $\That$ and $\That'_0$, respectively. By the construction of $T'_0$, the tree $T$ can be obtained from $T'_0$ by contracting the edge $e'_{i''}$. Hence, the fundamental cocycle of $e'_{i'}$ with respect to $T'_0$ is
	\[\opcoc(e'_{i'},T'_0)=\{ e'_{l'(j)} \mid  e_j \in \opcoc(e_i, T),\, j \neq i\} \cup \{e'_{i'}\}.\]
	Therefore, if $ \alpha_{T}(e_{i})=L$ (respectively, $D$), this implies that $e'_{i'}$ is (respectively, is not) the lowest-ordered edge in $\opcoc(e'_{i'}, T'_0)$. Thus we obtain $\alpha_{T}(e_{i})=\alpha_{T'_0}(e'_{i'})$. Similarly, if we change the roles of $e'_{i'}$, $e'_{i''}$, we obtain $\alpha_{T}(e_{i})=\alpha_{T'_0}(e'_{i''})$. Since the above hold for any ordering of $i'$ and $i''$, the activity letters $\alpha_{T'_0}(e'_{i'})$ and $\alpha_{T'_0}(e'_{i''})$ are independent of the order of $i'$ and $i''$.\\
	(b) Let $\Tbar'_1$ and $\Tbar'_2$ be the dual spanning tree parts of $\That'_1$ and $\That'_2$, respectively. Note that $\ebar'_{i'}$ and $\ebar'_{i''}$ are two edges between the same two vertices. By assumption, $i' \textless i''$; therefore, by part (c) of Lemma~\ref{lem:bigonexcontro}, the activity letter satisfies $\alpha_{\Tbar'_1}(\ebar'_{i''})=D$.  On the other hand, by the construction of $\Tbar'_2$, the tree $\Tbar$ can be obtained from $\Tbar'_2$ by replacing $\ebar'_{i'}$ with $\ebar_i$. Hence, the fundamental cocycle of $\ebar'_{i'}$ with respect to $\Tbar'_2$ is
	\[\opcoc(\ebar'_{i'},\Tbar'_2)=\{ \ebar'_{l'(j)} \mid  \ebar_j \in \opcoc(\ebar_i, \Tbar),\, j \neq i\} \cup \{\ebar'_{i'}\} \cup \{\ebar'_{i''}\}.\]
	Therefore, if $\alpha_{\Tbar}(\ebar_{i})=L$ (respectively, $D$), this implies that $\ebar'_{i'}$ is (respectively, is not) the lowest-ordered edge in $\opcoc(\ebar'_{i'},\Tbar'_2)$. Thus we obtain $\alpha_{\Tbar}(\ebar_{i})= \alpha_{\Tbar'_2}(\ebar'_{i'})$. This also shows the dependency of the activity letters on the order of $i'$ and $i''$, and hence completes the proof.
\end{proof}

\begin{figure}[hbp]
	\centering
	\subfigure[]{
		\includegraphics[width=0.45\textwidth]{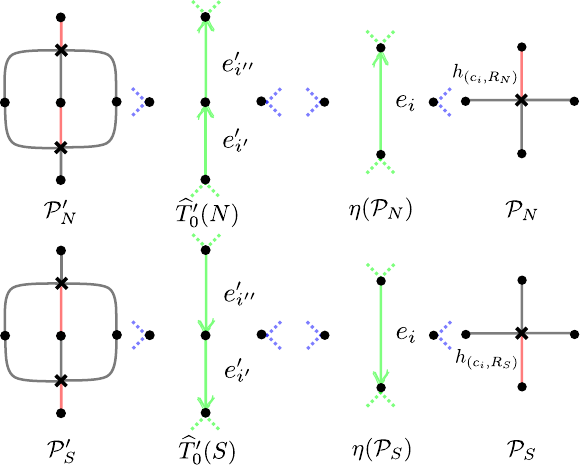}}
	\hspace{2mm}
	\subfigure[]{
		\includegraphics[width=0.45\textwidth]{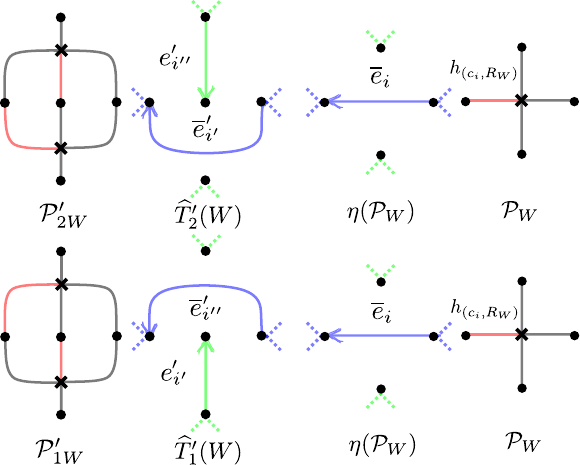}}
	\caption{(a) The correspondence between perfect matchings induced by the 1-to-1 correspondence of double rooted spanning trees. (b) The correspondence between perfect matchings induced by the 1-to-2 correspondence of double rooted spanning trees.}
	\label{fig:bigonextensionmatch}
\end{figure}

\begin{lemma}\label{lem:bigonexmain}
Let $L_l$ be a link universe with an ordered labeling $l$, and let $L'_{l'}$ be the link universe obtained from a bigon extension $E_{c_i}$. Let $\That$ be a double spanning tree in $G^d_L$, and let $\That'$ denote a corresponding double spanning tree of $\That$ after bigon extension. For any edge $\ehat_k \in \That$ with $k \neq i$, denote by $\ehat'_{l'(k)}$ its corresponding edge in $\That'$. Then the activity letters satisfy \[\alpha_{\That} \big(\ehat_k\big)=\alpha_{\That'} \big(\ehat'_{l'(k)}\big).\]
\end{lemma}
\begin{proof}
Without loss of generality, we assume $E_{c_i}$ is of type A; the case for a bigon extension of type B is similar. We consider separately the cases depending on whether $e_i$ belongs to $\That$, as it was shown in Proposition~\ref{prop:doublespancorre}.
\begin{itemize}
	\item [\text{Case 1:}] If $e_i \in \That$, let $\That'_0$ be the unique double spanning tree corresponding to $\That$. Then the following hold:
	\begin{itemize}
		\item [(a)] for $e_k \in T$ with $k \neq i$, the fundamental cocycle of $e'_{l'(k)}$ with respect to $T'_0$ is 
		\[\opcoc(e'_{l'(k)}, T'_0)=\{ e'_{l'(j)} \mid  e_j \in \opcoc(e_k, T),\, j \neq i\}.\]
		\item [(b)] for $\ebar_k \in \Tbar$ with $k \neq i$, the fundamental cocycle of $\ebar'_{l'(k)}$ with respect to $\Tbar'_0$ is described as follows:
		\begin{itemize}
		\item [(1)] if $\ebar_i \in \opcoc(\ebar_k, \Tbar)$, then
		\[\opcoc(\ebar'_{l'(k)},\Tbar'_0)=\{ \ebar'_{l'(j)} \mid  \ebar_j \in \opcoc(\ebar_k, \Tbar),\, j \neq i\} \cup \{\ebar'_{i'}\} \cup \{\ebar'_{i'}\}.\]
		\item [(2)] if $\ebar_i \notin \opcoc(\ebar_k, \Tbar)$, then
		\[\opcoc(\ebar'_{l'(k)},\Tbar'_0)=\{ \ebar'_{l'(j)} \mid  \ebar_j \in \opcoc(\ebar_k, \Tbar),\, j \neq i\}.\]
		\end{itemize}
	\end{itemize}
\item [\text{Case 2:}] If $\ebar_i \in \That$, let $\That'_1$ be the double spanning tree corresponding to $\That$ which contains $\ebar'_{i''}$.( the case for $\That'_2$ is similar.) Then the following hold:
	\begin{itemize}
		\item [(a)] for $e_k \in T$ with $k \neq i$, the fundamental cocycle of $e'_{l'(k)}$ with respect to $T'_1$ is described as follows:
		\begin{itemize}
		\item [(1)] if $e_i \in \opcoc(e_k, T_1)$, then
		\[\opcoc(e'_{l'(k)}, T'_1)=\{ e'_{l'(j)} \mid  e_j \in \opcoc(e_k, T),\, j \neq i\} \cup \{e'_{i''}\}.\]
		\item [(2)] if $e_i \notin \opcoc(e_k, T_1)$, then
		\[\opcoc(e'_{l'(k)}, T'_1)=\{ e'_{l'(j)} \mid  e_j \in \opcoc(e_k, T),\, j \neq i\}.\]
		\end{itemize}
		\item [(b)] for $\ebar_k \in \Tbar$ with $k \neq i$, the fundamental cocycle of $\ebar'_{l'(k)}$ with respect to $\Tbar'_1$ is
		\[\opcoc(\ebar'_{l'(k)},\Tbar'_1)=\{ \ebar'_{l'(j)} \mid  \ebar_j \in \opcoc(\ebar_k, \Tbar),\, j \neq i\}.\]
	\end{itemize}
\end{itemize}

Note that if $k \textgreater i$ (respectively, $k \textless i$), then $l'(k)$ is greater (respectively, lower) than both $i'$ and $i''$. If $\alpha_{\That} (\ehat_k)=D$, this implies that there exists a lower-ordered edge in the fundamental cocycle of $\ehat_k$. If such an edge is labeled $i$, then from the above discussion, there exists an edge labeled $i'$ or $i''$ in the fundamental cocycle of $\ehat'_{l'(k)}$. If such an edge is labeled $j$ with $j \neq i$, there exists an edge labeled $l'(j)$ in the fundamental cocycle of $\ehat'_{l'(k)}$. Hence, $\ehat'_{l'(k)}$ is not the lowest-ordered edge in its fundamental cocycle and $\alpha_{\That'} (\ehat'_{l'(k)})=D$. The case $\alpha_{\That} (\ehat_k)=L$ is similar. This completes the proof.
\end{proof}

\vspace{0.5em}
The following proposition describes how a bigon extension affects the activity letters of the half-edges, excluding those incident to the vertex corresponding to the newly created region.

\begin{proposition}\label{prop:bigonexpropcore}
	Let $L_{s; l}$ be a link configuration, and $L'_{s; l'}$ is obtained from $L_{s; l}$ by a type-A bigon extension $E_{c_i}$. Let $G^b_{L,s}$ and $G^b_{L',s}$ denote the balanced overlaid checkerboard graphs associated with $L_s$ and $L'_s$, respectively. The notation for the crossing points and regions is as shown in Figure~\ref{fig:bigonextension}. Then the following hold:
	\begin{itemize}
	\item [(a)] Let $c_j$ be a crossing point of $L$ with $j \neq i$, then the half-edge $h_{(c_j, R)} \in G^b_{L,s}$ satisfies the EI-property if and only if $h'_{(c'_{l'(j)}, R')} \in G^b_{L',s}$ satisfies the EI-property. Moreover, in this case, we have
	\[\alpha \big(h_{(c_j, R)}\big) =\alpha \big(h'_{(c'_{l'(j)}, R')}\big).\]
	\item [(b)] The half-edge $h_{(c_i, R_S)}$ satisfies the EI-property if and only if $h'_{(c'_{i'}, R'_S)}$ satisfies the EI-property. Similarly, the half-edge $h_{(c_i, R_N)}$ satisfies the EI-property if and only if $h'_{(c'_{i''}, R'_N)}$ satisfies the EI-property. Moreover, in this case, we have
	\[\alpha \big(h_{(c_i, R_S)}\big) =\alpha \big(h'_{(c'_{i'}, R'_S)}\big) \, \text{ and } \, \alpha \big(h_{(c_i, R_N)}\big) =\alpha \big(h'_{(c'_{i''}, R'_N)}\big).\] 
	\item [(c)] Let $R_d$ denote a region belonging to $\{R_W, R_S\}$. Suppose $i' \textless i''$. Then the half-edge $h_{(c_{i}, R_d)}$ satisfies the EI-property if and only if  $h'_{(c'_{i'}, R'_d)}$ satisfies the EI-property. Moreover, in this case, we have
	\[\alpha \big(h_{(c_{i}, R_d)}\big) =\alpha \big(h'_{(c'_{i'}, R'_d)}\big).\]
	On the other hand, $h'_{(c'_{i''}, R'_d)}$ always satisfies the EI-property and $\alpha(h'_{(c'_{i''}, R'_d)})=D$.
	\end{itemize}
\end{proposition}
\begin{proof}
Note that the correspondence between double rooted spanning trees naturally induces a correspondence between perfect matchings, as illustrated in Figure~\ref{fig:bigonextensionmatch}.\\
(a) Let $\Pcal'_1$ and $\Pcal'_2$ be two perfect matchings of $G^b_{L',s}$ that both contain the half-edge $h_{(c'_{l'(j)}, R')}$. There exist two perfect matchings $\Pcal_1$ and $\Pcal_2$ of $G^b_{L,s}$ such that the associated double spanning trees $\eta(\Pcal'_1)$ and $\eta(\Pcal'_2)$ correspond to $\eta(\Pcal_1)$ and $\eta(\Pcal_2)$, respectively. Note that in each corresponding pair of double spanning trees, the full-edge $\ehat'_{l'(j)}$ and its counterpart $\ehat_j$ are assigned the same direction relative to their respective roots. It follows that both $\Pcal_1$ and $\Pcal_2$ contain the half-edge $h_{(c_j, R)}$, which corresponds to the full-edge $\ehat_j$. \\
Thus by Lemma~\ref{lem:bigonexmain}, the activity letters satisfy 
\[ \alpha_{\eta(\Pcal'_1)}\big(\eta(h'_{(c'_{l'(j)}, R')})\big) =   \alpha_{\eta(\Pcal_1)}\big(\eta(h_{(c_j, R)})\big)  \, \text{ and } \,   \alpha_{\eta(\Pcal_2)}\big(\eta(h_{(c_j, R)})\big)= \alpha_{\eta(\Pcal'_2)}\big(\eta(h'_{(c'_{l'(j)}, R')})\big). \]
If $h_{(c_j, R)}$ satisfies the EI-property, we have
\[\alpha_{\eta(\Pcal_1)}\big(\eta(h_{(c_j, R)})\big) = \alpha_{\eta(\Pcal_2)}\big(\eta(h_{(c_j, R)})\big).\]
It follows that $h'_{(c'_{l'(j)}, R')}$ also satisfies the EI-property, proving the necessity part of the proposition. The sufficiency follows by a similar argument, and the equality of the corresponding activity letters is immediate. This completes the proof.\\
(b) We only need to prove the first part of the proposition; the other part can be proven similarly. Let $\Pcal_S$ be a perfect matching containing $h_{(c_i, R_S)}$, then $\eta(\Pcal_S)$ is a double spanning tree containing $e_i$. Hence, $\eta(\Pcal_S)$ corresponds to a unique double spanning tree $\That'_0(S)$. Denote the perfect matching $\theta(\That'_0(S))$ by $\Pcal'_S$; see subfigure (a) of Figure~\ref{fig:bigonextensionmatch} for an illustration. The perfect matching $\Pcal'_S$ contains $h'_{(c'_{i'}, R'_S)}$, and all perfect matchings containing both $h'_{(c'_{i'}, R'_S)}$ and $h'_{(c'_{i''}, R'_C)}$ arise in this way. By part (a) of Lemma~\ref{lem:bigonexcorr}, the activity letters satisfy
\[\alpha_{\eta(\Pcal_S)}\big(\eta(h_{(c_i, R_S)})\big) = \alpha_{\eta(\Pcal_S)}(e_i) =\alpha_{\That'_0(S)}(e'_{i'})=\alpha_{\eta(\Pcal'_S)}\big(\eta(h'_{(c'_{i'}, R'_S)})\big).\]
It follows that $h_{(c_i, R_S)}$ satisfies the EI-property if and only if $h'_{(c'_{i'}, R'_S)}$ satisfies the EI-property, and the equality of the corresponding activity letters is immediate.\\
(c) Assume $R_d=R_W$; the case $R_d=R_E$ is similar. Let $\Pcal'_{1W} \subset G^b_{L',s}$ be a perfect matching containing $h'_{(c'_{i''}, R'_W)}$, then $\eta(\Pcal'_{1W})$ contains $\ebar'_{i''}$. By part (b) of Lemma~\ref{lem:bigonexcorr}, we have $\alpha_{\eta(\Pcal'_{1W})}(\ebar'_{i''})=D$. This shows that $h'_{(c'_{i''}, R'_d)}$ satisfies the EI-property and $\alpha\big(h'_{(c'_{i''}, R'_d)}\big)=D$. \\
Now, let $\Pcal_W$ be a perfect matching containing $h_{(c_{i}, R_W)}$, then $\eta(\Pcal_W)$ is a double spanning tree containing $\ebar_i$. Hence, $\eta(\Pcal_W)$ corresponds to a double spanning tree $\That'_2(W)$ that contains $\ebar'_{i'}$. Denote the perfect matching $\theta(\That'_2(W))$ by $\Pcal'_{2W}$; see subfigure (b) of Figure~\ref{fig:bigonextensionmatch} for an illustration. The perfect matching $\Pcal'_{2W}$ contains $h'_{(c'_{i'}, R'_W)}$, and all such matchings arise in this way. By part (b) of Lemma~\ref{lem:bigonexcorr}, the activity letters satisfy
\[\alpha_{\eta(\Pcal_W)}\big(\eta(h_{(c_{i}, R_W)})\big) = \alpha_{\eta(\Pcal_W)}(\ebar_i) = \alpha_{\That'_2(W)}(\ebar'_{i'}) =\alpha_{\eta(\Pcal'_{2W})}\big(\eta(h'_{(c'_{i'}, R'_W)})\big).\]
It follows that $h_{(c_{i}, R_d)}$ satisfies the EI-property if and only if  $h'_{(c'_{i'}, R'_d)}$ satisfies the EI-property, and the equality of the corresponding activity letters is immediate. This completes the proof.
\end{proof}

\begin{corollary}
	Let $L_{s; l}$ be a link configuration, suppose $L$ has a bigon whose incident vertices are two crossing points labeled by consecutive numbers. Let $L''$ be the link universe obtained from $L$ by a bigon reduction at this bigon, $l''$ be the natural ordered labeling induced by $l$. Then, if $L_{s; l}$ satisfies the \textit{EI-property}, $L''_{s; l''}$ also satisfies the \textit{EI-property}.
\end{corollary}
\begin{proof}
Let $G^b_{L,s}$ and $G^b_{L'',s}$ be the balanced overlaid checkerboard graph corresponding to $L_s$ and $L''_s$, respectively. Since $L''_{s; l''}$ is obtained from $L_{s; l}$ by a bigon reduction, $L_{s; l}$ can be obtained from $L''_{s; l''}$ by a bigon extension. By Proposition~\ref{prop:bigonexpropcore}, for each half-edge $h'' \in G^b_{L'',s}$, there exists a corresponding half-edge $h \in G^b_{L,s}$ such that if $h$ satisfies the \textit{EI-property}, then $h''$ also satisfies the \textit{EI-property}. Therefore, if $L_{s; l}$ satisfies the \textit{EI-property}, so does $L''_{s; l''}$.
\end{proof}

\noindent\textbf{Proof of Theorem~\ref{thm:localcri}.}
Without loss of generality, assume the bigon extension is of type A; the type B case is similar.\\
(a)We first address the sufficiency part. We claim that under these conditions, whatever the order of $i'$ and $i''$, one of the half-edges $h'_{(c'_{i'}, R'_C)}$ or $h'_{(c'_{i''}, R'_C)}$ will not satisfy the EI-property. Assume the contrary. Since $L$ is prime-like, by \cite[Theorem 4.7]{cohen2014kauffman}, each half-edge in $G^b_{L,s}$ is contained in some perfect matching. Since neither $v_S$ nor $v_N$ is an omitted vertex, there exist perfect matchings containing the half-edges $h_{(c_i, R_S)}$ and $h_{(c_i, R_N)}$, denoted by $\Pcal_S$ and $\Pcal_N$, respectively. Note that $\eta \big(h_{(c_i, R_S)}\big)$ and $\eta \big(h_{(c_i, R_N)}\big)$ are both $e_i$. By assumption, we have $\alpha \big(h_{(c_i, R_S)}\big)=\alpha \big(h_{(c_i, R_N)}\big)$, hence 
\[\alpha_{\eta(\Pcal_S)} \big(\eta(h_{(c_i, R_S)})\big)=\alpha_{\eta(\Pcal_S)}(e_i)=\alpha_{\eta(\Pcal_N)}(e_i)=\alpha_{\eta(\Pcal_N)} \big(\eta(h_{(c_i, R_N)})\big).\]  
Since $\eta(\Pcal_N)$ and $\eta(\Pcal_S)$ are both double spanning trees containing $e_i$, denote by $\That'_0(N)$ and $\That'_0(S)$ the double spanning trees corresponding to $\eta(\Pcal_N)$ and $\eta(\Pcal_S)$, respectively. By part (a) of Lemma~\ref{lem:bigonexcorr}, the activity letters satisfy 
\[\alpha_{\That'_0(S)} (e'_{i'})=\alpha_{\That'_0(S)} (e'_{i''})=\alpha_{\eta(\Pcal_S)} (e_i)=\alpha_{\eta(\Pcal_N)} (e_i)=\alpha_{\That'_0(N)} (e'_{i'})=\alpha_{\That'_0(N)} (e'_{i''}),\]
which is independent of the order of $i'$ and $i''$. Note that in the perfect matching $\Pcal'_N=\theta (\That'_0(N))$, we have $\theta (e'_{i'}) =h'_{(c'_{i'}, R'_C)}$; and in $\Pcal'_S=\theta (\That'_0(S))$, we have $\theta (e'_{i''})=h'_{(c'_{i''}, R'_C)}$, as illustrated in Figure~\ref{fig:bigonextensionmatch}. Therefore, if $h'_{(c'_{i'}, R'_C)}$ and $h'_{(c'_{i''}, R'_C)}$ both satisfy the EI-property, it must follow that
\[\alpha \big(h'_{(c'_{i'}, R'_C)}\big)=\alpha \big(h'_{(c'_{i''}, R'_C)}\big).\]

Now, consider other perfect matchings containing $h'_{(c'_{i'}, R'_C)}$ and $h'_{(c'_{i''}, R'_C)}$. Since at least one of $v_W$ and $v_E$ is not an omitted vertex, we may assume, without loss of generality, that it is $v_W$. Then there exists a perfect matching containing the half-edge $h_{(c_i, R_W)}$, denoted by $\Pcal_W$. Since the double spanning tree $\eta(\Pcal_W)$ contains the full-edge $\ebar_i$, denote by $\That'_1(W)$ and $\That'_2(W)$ the two double spanning trees corresponding to $\eta(\Pcal_W)$ that contain $e'_{i'}$ and $e'_{i''}$, respectively. By parts (a) and (b) of Lemma~\ref{lem:bigonexcontro}, regardless of the order of  $i'$ and $i''$, the activity letter satisfies 
\[\alpha_{\That'_1(W)}(e'_{i'}) \neq \alpha_{\That'_2(W)}(e'_{i''}).\]
Note that in the perfect matching $\Pcal'_{1W}=\theta(\That'_1(W))$, we have $\theta(e'_{i'})=h'_{(c'_{i'}, R'_C)}$; and in $\Pcal'_{2W}=\theta(\That'_2(W))$, we have $\theta(e'_{i''})=h'_{(c'_{i''}, R'_C)}$. Hence, the edges $\eta\big(h'_{(c'_{i'}, R'_C)}\big)$ and $\eta\big(h'_{(c'_{i''}, R'_C)}\big)$ have different activity letters with respect to $\eta(\Pcal'_{1W})$ and $\eta(\Pcal'_{2W})$, respectively. This leads to a contradiction, which proves our claim.

Now, we turn to the necessity part. To do this, we need to show that, if either one of the two conditions is not satisfied, then there exists a natural ordered labeling $l'$ such that every half-edge in $G^b_{L'}$ satisfies the EI-property. By Proposition~\ref{prop:bigonexpropcore}, it suffices to verify this for $h'_{(c'_{i'}, R'_C)}$ and $h'_{(c'_{i''}, R'_C)}$. Since at least one of $v_S$ and $v_N$ is not an omitted vertex, we may assume without loss of generality that it is $v_S$. As illustrated in Figure~\ref{fig:bigonextensionmatch}, there are two types of perfect matchings containing $h'_{(c'_{i''}, R'_C)}$: one corresponding to $\Pcal_S \subset G^b_{L,s}$ containing $h_{(c_i, R_S)}$ and another corresponding to some $\Pcal \subset G^b_{L,s}$ containing $h_{(c_i, R_W)}$ or $h_{(c_i, R_E)}$. For the first case, by the discussion in the sufficiency part, the activity letters satisfy
\[\alpha_{\That'_0(S)} \big(\eta(h'_{(c'_{i''}, R'_C)})\big)=\alpha_{\That'_0(S)} (e'_{i''})=\alpha_{\eta(\Pcal_S)} (e_i)=\alpha\big(h_{(c_i, R_S)}\big).\]
Also, for the second case, the activity letters satisfy
\[\alpha_{\That'_2(W)}(e'_{i''})=\alpha_{\That'_2(W)} \big(\eta(h'_{(c'_{i''}, R'_C)})\big)=L \, \text{ if }  i' \textgreater i''; \text{ and }  \alpha_{\That'_2(W)} \big(\eta(h'_{(c'_{i''}, R'_C)})\big)=D \, \text{ if } i' \textless i''.\]
Hence, if $\alpha\big(h_{(c_i, R_S)}\big)=D$, choose $i' \textless i''$; and if $\alpha\big(h_{(c_i, R_S)}\big)=L$, chosse $i' \textgreater i''$. This ensures that $h'_{(c'_{i''}, R'_C)}$ satisfies the EI-property. Note that there is a similar result for $h'_{(c'_{i'}, R'_C)}$. Since either $v_N$ is an omitted vertex or $\alpha\big(h_{(c_i, R_N)}\big) \neq \alpha\big(h_{(c_i, R_S)}\big)$, the above labeling choice also ensures that $h'_{(c'_{i'}, R'_C)}$ satisfies the EI-property. This completes the proof.\\
(b) Note that by part (a), if a bigon extension is admissible, then the activity letters satisfy $\alpha\big(h_{(c_i, R_N)}\big) \neq \alpha\big(h_{(c_i, R_S)}\big)$. Hence, the labeling choice specified in the statement is well defined. Also note that, in the necessity part of the proof of (a), we showed that such a labeling choice ensures that $L'_{s; l'}$ satisfies the EI-property. Therefore, it suffices to prove the uniqueness of this choice. Suppose we interchange $i'$ and $i''$. By Lemma~\ref{lem:bigonexcorr}, the activity letters of $e'_{i'}$ and $e'_{i''}$ remain unchanged in the spanning trees that contain both $e'_{i'}$ and $e'_{i''}$. But they change in spanning trees that contain only one of these edges. This implies that interchange $i'$ and $i''$ would cause either $h'_{(c'_{i'}, R'_C)}$ or $h'_{(c'_{i''}, R'_C)}$ to violate the EI-property. This establishes the uniqueness of the labeling choice.
\qed

\vspace{0.5em}
We introduce a special class of crossing points at which both types of admissible bigon extensions can be performed iteratively.

\begin{definition}
Let $L_{s; l}$ be a link configuration satisfying the \textit {EI-property}, a crossing point $c$ in $L$ is said to be \textbf{active} if, among the half-edges $h \in G^b_L$ incident to the vertex $v_c$, there are exactly two half-edges forming a corner at the $v_c$ such that each of them has activity letter $L$.
\end{definition}

\begin{proposition}\label{prop:actcross}
Let $L_{s; l}$ be a link configuration satisfying the \textit {EI-property}, and $L'_{s; l'}$ is obtained from $L_{s; l}$ by a bigon extension $E_{c_i}$. Then the following properties hold:
\begin{itemize}
	\item [(a)] Let $c_j$ with $j \neq i$ be an active crossing point of $L$. If the bigon extension $E_{c_i}$ is admissible, then $c'_{l'(j)}$ is an active crossing point of $L'$.
	\item [(b)] If $c_i$ is an active crossing point of $L$, then the bigon extension $E_{c_i}$ is always admissible, and $c'_i$ is an active crossing point of $L'$.
\end{itemize}
\end{proposition}
\begin{proof}
(a) This follows immediately from part (a) of Proposition~\ref{prop:bigonexpropcore}.\\
(b) Let $\ehat_c$ be the edge extended by $E_{c_i}$. Since $c_i$ is active, either (i) one of the endpoints $\ehat_c$ is an omitted vertex, or (ii) the two half-edges corresponding to $\ehat_c$ are assigned different activity letters. In either case, the extension $E_{c_i}$ satisfies the admissibility condition. The fact that $c'_i$ is active then follows directly from part (c) of Proposition~\ref{prop:bigonexpropcore} and the argument presented in the proof of Theorem~\ref{thm:localcri}.
\end{proof}

We introduce a special class of link configurations satisfying the \textit{EI-property}, in which the activity letters of half-edges are easily determined.

\begin{definition}
A link configuration $L_{s; l}$ is said to be \textbf{admissible} if it can be obtained from the link configuration $H_{s; l_H}$ by a sequence of admissible bigon extensions, where $H$ is the Hopf link universe and $l_H$ is an ordered labeling of $H$.
\end{definition}

\begin{proposition}\label{prop:activityrule}
Let $L_{s; l}$ be an admissible link configuration, then $L_{s; l}$ satisfies the \textit{EI-property}. Moreover, the activity letter of each half-edge $h_{(c_j, R)} \in G^b_{L,s}$ is determined by the following rule: if $c_j$ is the crossing point with the lowest label among all the crossing points incident to $R$, then $\alpha \big(h_{(c_j, R)}\big)=L$. Otherwise, we have $\alpha \big(h_{(c_j, R)}\big)=D$.
\end{proposition}
\begin{proof}
By Proposition~\ref{prop:HopfEI}, the link configuration $H_{s; l_H}$ always satisfies the EI-property. Since $L_{s; l}$ is obtained from $H_{s; l_H}$ by a sequence of admissible bigon extensions, and each such extension preserves the EI-property, it follows that $L_{s; l}$ also satisfies the EI-property.  This completes the proof of the first part of the proposition.

We prove the second part of the proposition by induction. Note that the checkerboard graph $G_H$ of $H$, as well as its dual $\Gbar_H$, each consists of two edges between the same two vertices. It follows that the edge labeled 1 is always active, while the edge labeled 2 is always inactive. Consequently, for the half-edges in $G^b_{H,s}$, the activity letters satisfy $h_{(c_j, R)}=L$ if $j=1$ and $h_{(c_j, R)}=D$ if $j=2$. Therefore, the statement holds for the base case $H_{s; l_H}$. Also note that $H$ is prime-like.

Now, assume the statement holds for a link configuration $L_{s; l}$ that satisfies the EI-property and in which $L$ is prime-like. Let $L'_{s;l'}$ be obtained from $L_{s; l}$ by an admissible bigon extension $E_{c_i}$. By Proposition~\ref{prop:prime}, $L'$ is also prime-like. Without loss of generality, we assume that $E_{c_i}$ is of type A. We adopt the notation for regions as depicted in Figure~\ref{fig:bigonextension}.

For a half-edge $h'_{(c'_j, R')} \in G^b_{L',s}$, where $R'$ is the newly created region $R'_C$, it follows from the proof of part (a) of Theorem~\ref{thm:localcri} that the activity letters satisfy 
\[\alpha \big(h'_{(c'_i, R'_C)} \big)=L \, \text{ and } \, \alpha \big(h'_{(c'_{i+1}, R'_C)} \big)=D,\] which obey the prescribed rule.

Now, consider the region $R' \in \{R'_S, R'_N\}$. We will prove the case $R'_S$; the argument for $R'_N$ is analogous. By parts (a) and (b) of Proposition~\ref{prop:bigonexpropcore}, it suffices to show that if $\alpha \big(h_{(c_{j}, R_S)} \big)=L$, then $l'(j)$ is the lowest label among all the labels of crossing points incident to $R'_S$. If $\alpha \big(h_{(c_{i}, R_S)} \big)=L$, by part (b) of Theorem~\ref{thm:localcri}, the crossing point $c'_{i'}$ is labeled $i+1$. By assumption, each crossing point $c_k$ incident to $R_S$ is labeled $k$ with $k \textgreater i$. Hence, the crossing point $c'_{l'(k)}$ is labeled $k+1$ in $L'$. Thus, $i+1$ is the lowest label among all such $k+1$. On the other hand, if $\alpha \big(h_{(c_{i}, R_S)} \big)=D$, then $c'_{i'}$ is labeled $i$. By assumption, there exists a crossing point $c_k$ with $k \textless i$ such that $\alpha \big(h_{(c_{k}, R_S)} \big)=L$. Therefore, $k$ is still the lowest label among all crossing points incident to $R'_S$. This completes the verification for $R'_S$.

Now, consider the region $R' \in \{R'_W, R'_E\}$. We will prove the case $R'_W$; the argument for $R'_E$ is analogous. By part (c) of Proposition~\ref{prop:bigonexpropcore}, we have 
\[\alpha \big(h_{(c_{i}, R_W)}\big) =\alpha \big(h'_{(c'_{i}, R'_W)}\big) \, \text{ and } \, \alpha \big(h'_{(c'_{i+1}, R'_W)}\big)=D.\] 
If $\alpha \big(h_{(c_{i}, R_W)}\big)=L$, then by assumption, each crossing point $c_k$ incident to $R_W$ is labeled $k$ with $k \textgreater i$. Therefore, the crossing point $c'_{l'(k)}$ is labeled $k+1$ in $L'$. Thus, $i$ is the lowest label among all such $k+1$, and we indeed have $\alpha \big(h'_{(c'_{i}, R'_W)}\big)=L$. On the other hand, if $\alpha \big(h_{(c_{i}, R_W)}\big)=D$, then by assumption, there exists a crossing point $c_k$ with $k \textless i$ such that $\alpha \big(h_{(c_{k}, R_W)} \big)=L$. Hence, $k$ remains the lowest label among all crossing points incident to $R'_S$. And we also have 
\[\alpha \big(h'_{(c'_k, R'_W)}\big)=L \, \text{ and } \, \alpha \big(h'_{(c'_{i}, R'_W)}\big)=D,\] 
which obey the prescribed rule. This completes the verification for $R'_W$.

Now, consider a region $R'$ distinct from those previously discussed and not adjacent to $s$. By assumption, there exists a crossing point $c_k$ such that $\alpha \big(h_{(c_{k}, R)}\big)=L$. The crossing point $c'_{l'(k)}$ of $L'$ is labeled $k$ if $k \textless i$, and $k+1$ if $k \textgreater i$. In either case, $c'_{l'(k)}$ is indeed the crossing point with the lowest label among all those incident to $R'$. By part (a) of Proposition~\ref{prop:bigonexpropcore}, we also have $\alpha \big(h'_{(c'_{l'(k)}, R')}\big)=L$. This completes the proof.
\end{proof}

We are now prepared to prove the first main theorem.

\begin{theorem}\label{thm:perexpan}
Let $\Ltil$ be a link diagram, and let $L$ be its corresponding link universe. If there exists a distinguished segment $s$ and an ordered labeling $l$ such that the link configuration $L_{s; l}$ is admissible, then the Kauffman bracket polynomial of $\Ltil$ admits a perfect matching expansion:
\begin{equation}\label{bracketper}
	\Gamma_{G^b_{L, s}}(A)=\sum_{\Pcal}\prod_{h \in \Pcal} \altil (h)|_K,
\end{equation}
where the sum is taken over all perfect matchings of the balanced overlaid checkerboard graph $G^b_{L, s}$, and the activity letter $\alpha (h)$ is assigned according to the rule described in Proposition~\ref{prop:activityrule}.
\end{theorem}
\begin{proof}
According to Proposition~\ref{prop:activityrule}, the admissible link configuration $L_{s; l}$ satisfies the EI-properrty, and the activity letter associated to each half-edge $h \in G^b_{L, s}$ is determined by the rule described there. Hence, the result follows directly from Proposition~\ref{prop:HopfEI}.
\end{proof}

\begin{example}
We present an example and a counterexample to illustrate the application of the perfect matching expansion of the Kauffman bracket polynomial.
\begin{itemize}
	\item [(a)] Consider the Whitehead link diagram $\Ltil$ with five crossing points. An ordered labeling $l$ and a distinguished segment $s$ are shown in subfigure (a) of Figure~\ref{fig:example1}, and each edge in $G^b_{L,s}$ is labeled with signed activity letters according to the rules in Proposition~\ref{prop:activityrule}. The weight associated with each perfect matching is already listed in subfigure (b) of Figure~\ref{fig:stateslattice}. Summing over these weights, we obtain the Kauffman bracket polynomial of $\Ltil$:
	\[\left\langle \Ltil \right\rangle=A^{-9}-A^{-5}+2A^{-1}-A^3+2A^7-A^{11}.\]
	\item (b) A counterexample illustrating a non-admissible type-A bigon extension $E_{c_3}$ is shown in subfigure (b) of Figure~\ref{fig:example1}. In this case, neither endpoint of the edge $e_3$, which is extended by $E_{c_3}$, is an omitted vertex. And both half-edges corresponding to $e_3$ are assigned the activity letter $D$, according to the rules in Proposition~\ref{prop:activityrule}. Consequently, by Theorem~\ref{thm:localcri}, this bigon extension is not admissible.
\end{itemize}
 
\end{example}

\begin{figure}[hbp]
	\centering
	\subfigure[]{
		\includegraphics[width=0.57\textwidth, height=0.11\textheight]{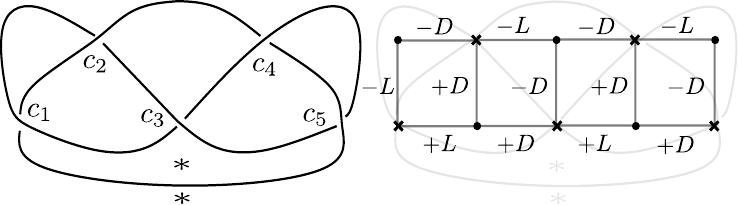}}
	\hspace{2mm}
	\subfigure[]{
		\includegraphics[width=0.37\textwidth,, height=0.10\textheight]{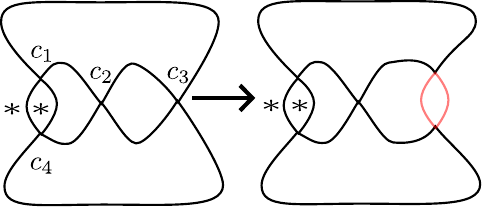}}
	\caption{(a) An example of the Whitehead link. (b) A counterexample of a non-admissible bigon extension, drawn in red.}
	\label{fig:example1}
\end{figure}

\subsection{Admissible link configurations}
We begin by introducing some constructions adapted from \cite{kauffman1983formal}; for the original and complete exposition, see the reference therein.

\begin{definition}
Let $L_s$ be a link universe with a distinguished segment $s$. Suppose $s$ is adjacent to the unbounded planar region. A \textbf{string} is obtained from $L_s$ by removing an iterior point from the segment $s$. The resulting diagram is drawn in the plane so that the two new endpoints of $s$ extend horizontally, one to the left and the other to the right.
\end{definition}

\begin{definition}\label{def:strings}
	We introduce the following strings:
	\begin{itemize}
		\item [(a)] The \textit{shell string} $S_0$ is obtained from $H_s$ where $H$ is the Hopf link universe with $s$ adjacent to the unbounded planar region. In this paper, we always assume that the black region $R^B_s$ lies below the string. We denote the upper arc of $S_0$ by $\bm{\gamma}_+$ and the lower arc by $\bm{\gamma}_-$. Let $\{q^0, q^{\infty} \}=\bm{\gamma}_+ \cap \ \bm{\gamma}_-$ be the two intersection points, where $q^0$ lies on the left and $q^{\infty}$ on the right. The common segment between the arcs, which lies within the circular region, is denoted by $\bm{\gamma}_0$, as shown in subfigure (a) of Figure~\ref{fig:threestrings}.
		\item [(b)] The \textit{signed curl string} is a curl, as shown in subfigure (b)
		of Figure~\ref{fig:threestrings}. The positive curl string, denoted by $+C_j$, is a curl where the loops lie above the string and has a marked point $+p_j$ on its final segment. Similarly, the negative curl string, denoted by $-C_j$ is a curl where the loops lie below the string and has a marked point $-p_j$ on its final segment.
		\item[(c)] The \textit{signed shell string} is a shell string, as shown in subfigure (c)
		of Figure~\ref{fig:threestrings}. The positive shell string, denoted by $+S_j$, has a marked point $+p_j^C$ on its upper arc $\bm{\gamma}^j_+$. Similarly, the negative shell string, denoted by $-S_j$ , has a marked point $-p_j^C$ on the lower arc $\bm{\gamma}^j_-$.
	\end{itemize}
\end{definition}

\begin{figure}[hbp]
	\centering
	\subfigure[]{
		\includegraphics[width=0.18\textwidth, height=0.1\textheight]{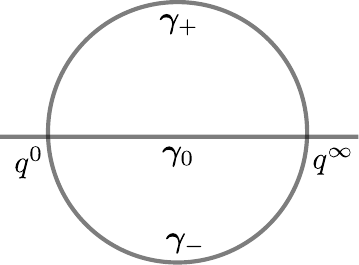}}
	\hspace{2mm}
	\subfigure[]{
		\includegraphics[width=0.2\textwidth, height=0.1\textheight]{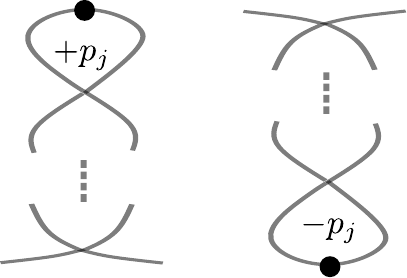}}
	\hspace{3mm}
	\subfigure[]{
		\includegraphics[width=0.35\textwidth, height=0.1\textheight]{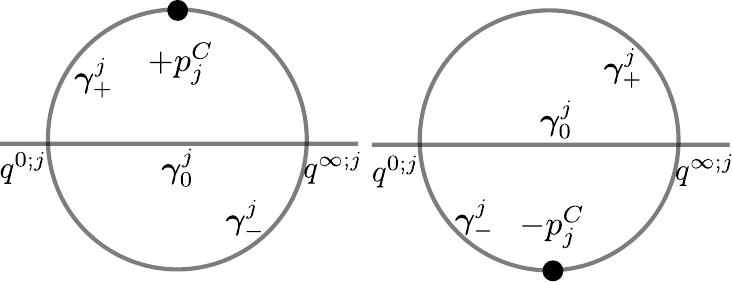}}
	
	\caption{(a) The base string $S_0$. (b) The positive and negative curl strings $+C_j$ and $-C_j$. (c) The positive and negative shell strings $+S_j$ and $-S_j$.}
	\label{fig:threestrings}
\end{figure}

\begin{definition}[Attach points]\label{def:points}
We attach various types of marked points to shell strings as follows:
\begin{itemize}
\item [(a)] For the shell string $S_0$, we first attach a sequence of points to the arc $\bm{\gamma}_0$. These points are linearly arranged from left to right in increasing order and come in three types: the trivial type, denoted by $p_k$; the curl-type, denoted by $p_k^C$; and the shell-type, denoted by $p_k^S$. Each point on $\bm{\gamma}_0$ is assigned a sign from $\{+, -\}$. For each point on $\bm{\gamma}_0$ labeled $k$ with sign $+$ (respectively, sign $-$), we attach a corresponding point $q^k$ on the upper arc $\bm{\gamma}_+$ (respectively, the lower arc $\bm{\gamma}_-$). The points on both $\bm{\gamma}_+$ and $\bm{\gamma}_-$ are also arranged linearly from left to right in increasing order. See the left part of Figure \ref{fig:ASIprocedure} for an example.
\item [(b)] For the shell string $S_j$, attach a sequence of points to the arc $\bm{\gamma}^j_0$. These points are linearly arranged from left to right in increasing order and are all of the trivial type, denoted by $p^j_k$. Each point on $\bm{\gamma}^j_0$ is assigned a sign from $\{+, -\}$. For each such point, attach a corresponding point $q^{k;j}$ on $\bm{\gamma}^j_+$ and $\bm{\gamma}^j_-$ as in the case of $S_0$. Furthermore, all points on the upper arc $\bm{\gamma}^j_+$ (respectively, the lower arc $\bm{\gamma}^j_-$) must lie to the left of the marked point $+p_j^C$ (respectively, $-p_j^C$).
\end{itemize}
\end{definition}

\begin{definition}[String composition]
The following operations on strings were introduced by Kauffman in \cite{kauffman1983formal}; for illustrations, see subfigures (a) and (b) of Figure~\ref{fig:stringsum}.
\begin{itemize}
\item [(a)] If $A$ and $B$ are strings, their \textit{sum} $A \oplus B$ is obtained by splicing the right line of $A$ to the left line of $B$.
\item [(b)] If $A$ and $B$ are strings, $p$ is a point on a non-connecting edge of $A$, and $p$ is not a crossing point, then the \textit{enclosed sum} $A \oplus [B,p]$ of $A$ and $B$ is obtained by replacing the trivial string at $p$ by $B$. In this situation, the string $A$ is called a $carrier$ and the string $B$ is called a rider in $A \oplus [B,p]$.
\end{itemize}
\end{definition}

\begin{definition}[Interaction]
Let $D$ be the diagram obtained from a sequence of string compositions. Let $(p,q)$ be a pair of points on $D$ such that there exists a region incident to both $p$ and $q$. An \textit{interaction} at $(p,q)$ is defined as the operation to glue the points $p$ and $q$ together to form a new crossing point, denoted by $c(q)$. see subfigure (c) of Figure~\ref{fig:stringsum} for an illustration.
\end{definition}

\begin{figure}[htp]
	\centering
	\includegraphics[width=0.8\textwidth]{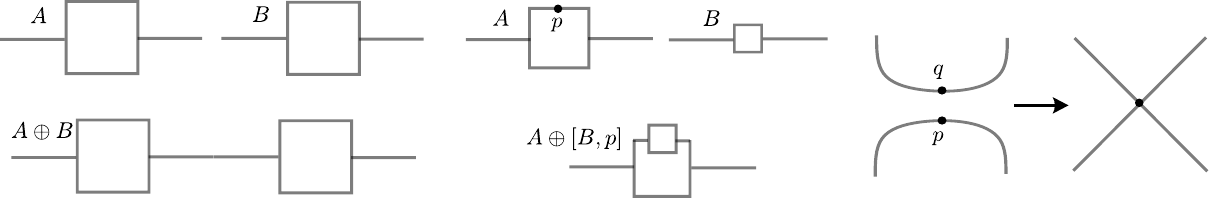}
	\caption{(a) The sum of two strings. (b) The enclosed sum of two strings. (c) The interaction between two points.}
	\label{fig:stringsum}
\end{figure}

\begin{definition}
Let the strings be as in Definition~\ref{def:strings}. A link configuration can be constructed by a process called \textbf{attaching points-string composition-interaction procedure} (abbreviated as the \textbf{ASI procedure}), which proceeds in the following steps:
\begin{itemize}
	\item [(i)] \textbf{Attach points}\\
	Attach a collection of marked points to the shell strings $S_0$ and each $S_j$ according to the rules given in Definition~\ref{def:points}.
	\item [(ii)] \textbf{String composition}\\
	The string composition is performed in two stages:
	\begin{itemize}
	\item [(1)]	Shell Attachment:\\
	Use the set of shell strings $S_i$ as riders and perform a sequence of enclosed sums on the base shell $S_0$ at the set of points $p_i^S$ (of shell type) on $\bm{\gamma}_0$, forming the intermediate diagram:
	\[D'= S_0 \oplus [+S_j,\ +p_j^S]\cdots \oplus [-S_k,\ -p_k^S] \cdots. \]
	\item [(2)] Curl Attachment:\\
	Next, use the set of curl strings $C_i$ as riders, perform a sequence of enclosed sums on the diagram $D'$ at the points $p_i^C$ (of curl type), forming the final diagram:
	\[D= D' \oplus [+C_j,\ +p_j^C]\cdots \oplus [-C_k,\ -p_k^C] \cdots.\] 
	\end{itemize}
	\item [(iii)] \textbf{Interaction}\\
	Finally, perform interactions at all specified pairs of points to reconstruct the crossings. Specifically, for each index $i$, apply interactions at the pairs $(p_i, q^i)$; and for each shell string $S_j$ apply interactions at the pairs $(p^j_i, q^{i;j})$. 
\end{itemize}
Thus, completing the ASI procedure by identifying the two ends of the string along the segment $s$, we obtain a link universe $L$ with distinguished segment $s$. Suppose $L$ has $n$ crossing points, the ordered labeling $l$ is defined according to the following rules:
\begin{itemize}
	\item [-] The crossing points $c(q^0)$ and $c(q^{\infty})$ are assigned the smallest and largest labels, respectively:
	\[l(c(q^0))=1 \text{, and } l(c(q^{\infty}))=n.\]
	\item [-] Let $c(p^C_j)$ denote the crossing point corresponding to the curl-type point $p^C_j$. Then the labels satisfy:
	\[l(c(q^{j-1})) \textless l(c(q^{0;j})) \leq l(c(q^{\infty;j})) \leq l(c(p^C_j)) \leq l(c(q^{j})) \textless l(c(q^{0;j+1}))\]
	\item [-] For the crossing points arising from the shell string $S_j$, the crossing corresponding to $q^{0;j}$ is assigned the minimal label, and the one corresponding to $q^{\infty;j}$ the maximal. The remaining crossings are labeled in increasing order from left to right according to their associated points $q^{i;j}$.
	\item [-] For the crossing points arising from the curl string $C_j$, the crossing point corresponding to $p_j^C$ is assigned the minimal label, and that corresponding to $q_j$ the maximal. The remaining crossings are labeled in increasing order according to their position along the curl, with crossings nearer to $q_j$ receiving larger labels.
\end{itemize}
\end{definition}

\begin{example}
An illustrative example of the ASI procedure is shown in Figure~\ref{fig:ASIprocedure}. In this example, the set of attached points on the base string $S_0$ is $[+p_1, -p_2, +p^S_3, -p^C_4]$, and the set of attached points on the signed shell string $+S_3$ is $[+p^3_1, -p^3_2]$. A 1-component curl string $+C_3$ is attached to the point $+p^C_3$, and a 2-component curl string $-C_4$ is attached to $-p^C_4$. Note that attaching a 0-component curl string is equivalent to replacing a curl-type marked point with a trivial-type marked point; this is permitted when the curl-type marked point lies on a signed shell string $S_j$.
\end{example}

\begin{figure}[htp]
	\centering
	\includegraphics[width=1\textwidth]{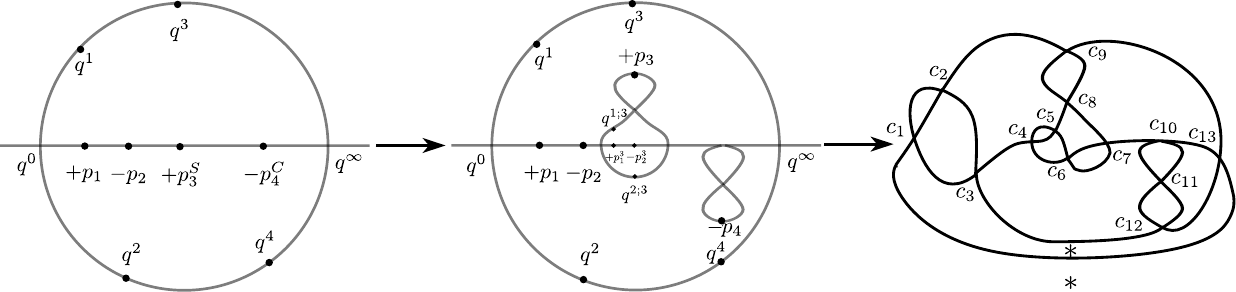}
	\caption{An illustrative example of the ASI procedure.}
	\label{fig:ASIprocedure}
\end{figure}

We now show that link configurations constructed via the ASI procedure fully characterize the class of admissible link configurations.

\begin{theorem}\label{thm:adlinktri}
Let $H^0_{s;l_H}$ denote the Hopf link configuration obtained from the base shell string $S_0$ (without attached points) via the ASI procedure. Then a link configuration $L_{s;l}$ is an admissible link configuration obtained from this specific configuration $H^0_{s;l_H}$ if and only if it can be constructed via the ASI procedure.
\end{theorem}
\begin{proof}
Since we always assume the black region $R^B_s$ lies below the base shell string $S_0$, the associated Hopf link configuration $H^0_{s;l_H}$ is uniquely determined. Let $L_{s;l}$ be a link configuration constructed via the ASI procedure. To prove the sufficiency, it remains to show that $L_{s;l}$ can be obtained from $H^0_{s;l_H}$ by a sequence of admissible bigon extensions. For the necessity, we must show that if $L'_{s;l'}$ is obtained from $L_{s;l}$ by a single admissible bigon extension, then $L'_{s;l'}$ can also be constructed via the ASI procedure.\\
We divide the discussion into three cases according to the type of marked points attached to $S_0$ in the construction of $L_{s;l}$:
\begin{itemize}
	\item [\text{Case 1.}] All points on $\bm{\gamma}_0$ are of the trivial type.
	\item [\text{Case 2.}] All points on $\bm{\gamma}_0$ are either of the trivial type or the curl type.
	\item [\text{Case 3.}] There exist points of the shell type on $\bm{\gamma}_0$.
\end{itemize}
\textbf{Case 1.} We prove the sufficiency by induction on the number of attached points. The base case is given by the Hopf link configuration $H^0_{s;l_H}$ by assumption. Let the set of attached points (ignoring signs) be denoted by $[p_1, \cdots, p_n]$, and suppose the corresponding link configuration $L_{s;l}$ is admissible. Now, attach a new point $p'_1$ to the left of $p_1$, so that the new set of attached points becomes $[p'_1, p'_2,\cdots, p'_{n+1}]$, where each $p'_{j+1}$ corresponds to $p_j$. Let $L^1_{s;l^1}$ denote the link configuration constructed from this new set of points. Observe that the crossing point $c_1$, which corresponds to $q^0$, is active. Thus, bigon extensions performed at $c_1$ are always admissible. If the new point $p'_1$ has a positive (respectively, negative) sign, then $L^1_{s;l^1}$ coincides with the link configuration $L'_{s;l'}$ obtained from $L_{s;l}$ by performing an admissible bigon extension of type A (respectively, type B) at $c_1$. This completes the induction step.

To prove the necessity, we begin by examining the points on the arcs $\bm{\gamma}_+$ and $\bm{\gamma}_- $, and classify the possible types of admissible bigon extensions that can be performed at these positions. Let R be a region enclosed within the base shell string $S_0$. By our convention,  $R$ is a black region if it is adjacent to the upper arc $\bm{\gamma}_+$; otherwise, it is a white region. Suppose $R$ intersects the arc it is adjacent to at two points $q^j$ and $q^k$ with $j \textless k$. Then, among all crossing points incident to $R$, the crossing $c(q^j)$ receives the minimal label. Therefore, the half-edge $h_{(c(q^j), R)}$ is assigned the activity letter $L$, while all other half-edges incident to $v_R$ receive the activity letter $D$. It follows that both types of admissible bigon extensions can be performed at each of these crossing points.

The bigon extensions at $c(q^0)$ have already been discussed in the sufficiency part; the case for $c(q^{\infty})$ is analogous. Let $c_{i+1}$ be a crossing point of $L$ corresponding to a point $q^{i}$ on the upper arc $\bm{\gamma}_+$ (the case for $\bm{\gamma}_-$ is similar). Suppose $L'_{s;l'}$ is obtained from $L_{s;l}$ by a type-A admissible bigon extension at $c_{i+1}$. Then, to prove necessity, it remains to show that $L'_{s;l'}$ can be constructed via the ASI procedure. The new set of attached point is given by \[[p'_1, \cdots, +p'_{i},+p'_{i+1}, \cdots,p'_{n+1}],\] 
where $+p'_{i}$ and $+p'_{i+1}$ both arise from subdividing the original point $+p_{i}$; this situation is illustrated in subfigure (i) of Figure~\ref{fig:ASIextension}. Note that, from the preceding discussion, the label of the crossing point corresponding to $q'^{i}$ satisfies $l'\big(c'(q'^{i})\big)=i+1$. Thus, the link configuration $L'_{s;l'}$ can be constructed from these new attached points via the ASI procedure. The case where the admissible bigon extension at $c_i$ is of type B will be discussed below in the Case 2.

\textbf{Case 2.} We first address the sufficiency part. Without loss of generality, assume there are $n$ attached points on the arc $\bm{\gamma}_0$, and exactly one of them is a point $+p^C_j$ of curl type with a positive sign (the case for a negative sign is similar), while the remaining points are all of trivial type. Let $L^2_{s;l^2}$ be the link configuration constructed from these n attached points, treating them all as of trivial type on the arc $\bm{\gamma}_0$. Suppose the corresponding curl string $+C_j$ is a 1-curl. Then we aim to show that the link configuration $L_{s;l}$, which incorporates $+p^C_j$, can be obtained from $L^2_{s;l^2}$ by a type-B admissible bigon extension at $c^2_{j+1}$. Note that, by assumption $c^2_{j+1}$ corresponds to the point $q^j$; this situation is illustrated in subfigure (ii) of Figure~\ref{fig:ASIextension}. Let $R^2$ be the region in $L^2$ that is both incident to $c^2_{j+1}$ and adjacent to the lower arc $\bm{\gamma}_-$. From the preceding discussion in case 1, we have $\alpha \big(h_{(c^2_{j+1},R^2)}\big)=D$. Therefore, in $L_{s;l}$, the label $l\big(c(p^C_j)\big)$ is strictly smaller than $l\big(c(q^j)\big)$, as required. This completes the proof of the necessity part of case 1. Note that the crossing point $c(p^C_j)$ is indeed active. Moreover, by our convention, the regions enclosed by the curl string $+C_j$ are white regions. Therefore, by part (b) of Proposition~\ref{prop:actcross}, we can perform type-B admissible bigon extensions at the crossing point labeled $l\big(c(p^C_j)\big)$ iteratively in $L_{s;l}$, thereby obtaining a curl string $+C_j$ with an increasing number of components.

To prove the necessity, we begin by classifying the possible types of admissible bigon extensions that can be performed at each crossing point. By part (a) of Proposition~\ref{prop:bigonexpropcore}, it suffices to examine the crossing points arising from the curl strings $+C_j$. By part (c) of Proposition~\ref{prop:bigonexpropcore}, all such crossing points, except for $c(p^C_j)$, only admit admissible bigon extensions of type B. These extensions serve solely to increase the number of components within the curl string. The case where an admissible bigon extension of type A is performed at $c(p^C_j)$ will be addressed separately in Case 3.

\textbf{Case 3.} We first address the sufficiency part. Let $[\cdots, +p^C_j, \cdots]$ be a set of attached points on $\bm{\gamma}_0$ that contains no points of shell type, and let $[\cdots, +p^S_j, \cdots]$ be the set obtained by replacing the point $+p^C_j$ with a shell-type point $+p^S_j$. Suppose the curl string $+C_j$ attached to the first set is an $n$-curl, and the curl string $+C_j$ attached to the second set is an $(n-1)$-curl. Additionally, assume the shell string $+S_j$ contains no attached points.

Let $L^3_{s;l^3}$ and $L_{s;l}$ be the link configurations constructed from the two sets of points, respectively. Our goal is to show that the link configuration $L_{s;l}$ can be obtained from $L^3_{s;l^3}$ by a type-A admissible bigon extension at the crossing point $c^3(p^C_j)$. It is clear that the link universe $L$ coincides with the one obtained from $L^3$ by performing the specified bigon extension; this situation is illustrated in subfigure (iii) of Figure~\ref{fig:ASIextension}. Let $R^3$ denote the black region to the right of the curl string $+C'_j$. By assumption, $c^3(p'^C_j)$ receives the minimal label among all crossing points incident to $R^3$. Therefore, in $L_{s;l}$, we have $l\big(c(q^{0;j})\big) \textless l\big(c(q^{\infty;j})\big)$, as required. This completes the proof of the necessity part of case 2. For a shell string $+S_j$ with attached points, note that by part (b) of Proposition~\ref{prop:actcross}, the crossing point $c(q^{0;j})$ is indeed active. Hence, in $L_{s;l}$, we can perform admissible bigon extensions of both types at $c(q^{0;j})$ iteratively, thereby obtaining the required link configuration. This completes the proof of the sufficiency part of Case 3.

The necessity part proceeds in a manner analogous to that of Case 1, with the only difference being that the points $p_i$ and $q^i$ are replaced by $p^j_i$ and $q^{i;j}$, respectively. Note that in this case, each crossing point $c(q^{i;j})$ on the upper arc $\bm{\gamma}^j_+$ admits only an admissible bigon extension of type B, while each crossing on the lower arc admits only an admissible bigon extension of type A. Thus, it suffices to examine the possible types of admissible bigon extensions that can be performed at the crossing point corresponding to $q^{\infty;j}$. Let $R^W$ and $R^B$ be the regions in $L$ lying outside the shell of $+S_j$ and incident to $c(q^{\infty;j})$. Since the white region $R^W$ is also incident to $c(q^{0;j})$, it follows that the crossing point $c(q^{\infty;j})$ admits only admissible bigon extensions of type A. This aligns with the rule that all points on the upper arc $\bm{\gamma}^j+$ must lie to the left of the marked point $+p_j^C$. This completes the proof.
\end{proof}

\begin{remark}\label{rem:inihopf}
For a Hopf link configuration $H_{s; l_H}$ with a different choice of distinguished segment or a different ordered labeling, analogous constructions can be carried out. For instance, if we reverse the order of the two crossing points in $H$, we may instead choose the white region $R^W_s$ to be the one lying below the base shell string $S_0$. Similarly, if the distinguished segment $s$ is not adjacent to the unbounded planar region, we may use strings whose connecting edges lie inside the shells.
\end{remark}

\begin{figure}[htp]
	\centering
	\begin{minipage}[t]{0.5\textwidth}
		\centering	
		\includegraphics[width=1.2\textwidth, height=0.15\textheight]{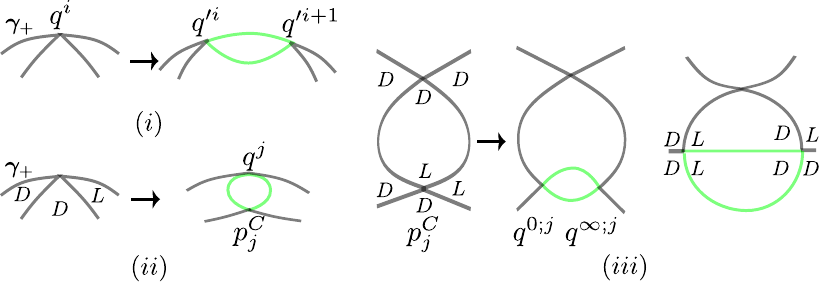}
		\caption{Illustrative examples of admissible bigon extensions arising in link configurations constructed via the ASI procedure.}
		\label{fig:ASIextension}
	\end{minipage}
	\hfill
	\begin{minipage}[t]{0.45\textwidth}
		\centering
		\includegraphics[width=0.62\textwidth, height=0.15\textheight]{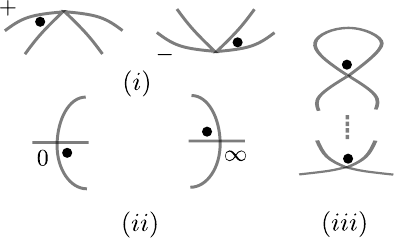}
		\caption{Markers corresponding to crossings in the minimal state.}
		\label{fig:statemini}
	\end{minipage}
\end{figure}

\begin{corollary}
Let $H_{s; l_H}$ be a fixed Hopf link configuration. If two admissible link configurations $L_{s; l}$ and $L_{s; l'}$ are obtained from $H_{s; l_H}$ and share the same link universe $L$ and distinguished segment $s$, then $l=l'$.
\end{corollary}
\begin{proof}
The conclusion follows immediately from Theorem~\ref{thm:adlinktri}, together with the observation in Remark~\ref{rem:inihopf}.
\end{proof}

\begin{corollary}\label{cor:actclass}
Let $L_{s; l}$ be an admissible link configuration constructed via the ASI procedure. Then a crossing point $c$ in $L$ is active if and only if it satisfies one of the following conditions:
\begin{itemize}
	\item [(i)] $c$ corresponds to the point $q^0$, which lies on the base shell string $S_0$.
	\item [(ii)] $c$ corresponds to the point $p^C_j$, which lies on a shell string $C_j$ attached to the arc $\bm{\gamma}_0$.
	\item [(iii)] $c$ corresponds to the point $q^{0;j}$, which lies on a shell string $S_j$ attached to the arc $\bm{\gamma}_0$.
\end{itemize}
\end{corollary}
\begin{proof}
This follows immediately from part (a) of Proposition~\ref{prop:actcross} and the discussion in the proof of Theorem~\ref{thm:adlinktri} .
\end{proof}

We next show that the weight of the minimal Kauffman state is determined by certain specific classes of regions.

\begin{definition}\label{def:regions}
Let $L_{s; l}$ be an admissible link configuration constructed via the ASI procedure. A region $R$ is said to be an \textit{upper region}, denoted by $R^u$, if it either lies adjacent to the upper arc of some shell string or is enclosed by a positive curl string. Otherwise, it is called a \textit{lower region}, denoted by $R^l$. We use the notation $\vert R^x_i \vert$ to denote the number of regions in $L$ belonging to a specified class, where $x \in \{u,l\}$ indicates whether the region is upper or lower, and $i \in \{1,2,3\}$ indexes the type of region as follows:
\begin{itemize}
	\item [$i=1$:] regions enclosed by the base shell string $S_0$ and lying outside both $C_j$ and $S_j$.
	\item [$i=2$:] regions enclosed by a curl string $C_j$.
	\item [$i=3$:] regions enclosed by a shell string $S_j$.
\end{itemize} 
\end{definition}

\begin{proposition}\label{pro:regiontype}
Let $L_{s; l}$ be an admissible link configuration constructed via the ASI procedure, and let $R$ be a region of $L$. Then both the color of $R$, together with the activity letter of the half-edge corresponding to $R$ that lies in the minimal perfect matching $\Pcal_{\min}$, are almost entirely determined by the specific region class to which $R$ belongs, except for the case of $R^l_1$, as summarized in Table~\ref{tab:regionclass}.
\end{proposition}
\begin{proof}
The markers corresponding to different types of crossing points in the minimal state $\Scal_{\min}$ are illustrated in Figure~\ref{fig:statemini}. Since we always assume the black region $R^B_s$ lies below the base shell string $S_0$, it follows that the regions adjacent to $\bm{\gamma}_+$ are black, while those adjacent to $\bm{\gamma}_-$ are white. The corresponding markers in $\Scal_{\min}$ for such regions are shown in subfigures (i) and (ii) of Figure~\ref{fig:statemini}. According to the rule governing the ordered labeling, if $(c,R^u_1) \in \Scal_{\min}$, there exists another crossing point $c'$ incident to $R^u_1$ with a lower label than $c$. Thus, we have $\alpha\big(h_{(c,R^u_1)}\big)=D$. Now, consider the region $R^l_1$. Suppose $(c,R^l_1) \in \Scal_{\min}$. if the crossing point $c$ corresponds either to a negative trivial-type attached point $-p_j$ or to the distinguished crossing point $c_1$, then $\alpha\big(h_{(c,R^l_1)}\big)=L$; otherwise, $\alpha\big(h_{(c,R^l_1)}\big)=D$. 

Next, we consider the regions enclosed by curl strings. It is clear that the upper regions enclosed by curl strings are white, while the lower regions enclosed by curl strings are black. The corresponding markers in $\Scal_{\min}$ for such regions are shown in subfigure (iii) of Figure~\ref{fig:statemini}. From this, we conclude that if $(c,R_2) \in \Scal_{\min}$, then the associated activity letter satisfies $\alpha\big(h_{(c,R_2)}\big)=L$. 

Furthermore, for regions enclosed by a shell string $S_j$, the activity letter is determined in the same manner as for those in class $R_1$. In particular, we observe that the region classes $R^u_3$ and $R^u_2$ share the same color, as do $R^l_3$ and $R^l_2$. This completes the proof.
\end{proof}

\begin{table}
\caption{Region colors and activity letters in the minimal state.}
\begin{tabularx}{\textwidth}{l *{6}{>{\centering\arraybackslash}X}}
	\toprule
	Region Class & $R^u_1$ & $R^l_1$ & $R^u_2$ & $R^l_2$ & $R^u_3$ & $R^l_3$ \\
	\midrule
	Region Color & B & W & W & B & W & B \\
	Activity Letter  & D & L or D & L & L & D & L \\
	\bottomrule
\end{tabularx}
\label{tab:regionclass}
\end{table}

We next show that certain classical link universes admit the structure of an admissible link configuration.

\begin{definition}\label{def:link}
We introduce several families of link universes that arise from classical link diagrams as follows:
\begin{itemize}
	\item [(a)] Let $(a_1,\ldots,a_n)$ be a sequence of positive integeres. The \textit{2-bridge link universe} $L(a_1,\cdots,a_n)$ refers to the link universe obtained from a standard 2-bridge diagram specified by the continued fraction data $(a_1,\cdots,a_n)$, as illustrated in subfigure (a) of Figure~\ref{fig:Links}. 
	\item [(b)] Let $(b_1,\ldots,b_n)$ be a sequence of positive integers. The \textit{pretzel link universe} $L(b_1,\ldots,b_n)$ is defined as the link universe corresponding to the classical pretzel diagram with twisting parameters $(b_1,\ldots,b_n)$, as illustrated in subfigure (b) of Figure~\ref{fig:Links}.
	\item [(c)] The \textit{Montesinos link universe} $L(k;R^1,\cdots,R^m)$ is a link universe of the form illustrated in subfigure (c) of Figure~\ref{fig:Links}, where k is a positive integer, and each $R^j(c^j_1,\ldots, c^j_n)$ denotes a rational tangle (disregarding over/under crossing information) parameterized by the sequence of integers $(c^j_1,\ldots, c^j_n)$, as depicted in the figure.
\end{itemize}
\end{definition}

\begin{proposition}
Let a link universe $L$ belong to one of the families introduced in Definition~\ref{def:link}, and let the lower segment be chosen as the distinguished segment $s$. Then there exists an ordered labeling $l$ of the crossing points of $L$, such that the associated pair $L_{s;l}$ forms an admissible link configuration. More precisely, for each family of link universes, the ordered labeling is determined as follows:
\begin{itemize}
	\item [(a)] Let $L$ be a 2-bridge link universe. Then the crossing points in the i-th twist box receive lower labels than those in the j-th box whenever $i \textless j$. Within each twist box, the crossing points are labeled in increasing order from left to right.
	\item [(b)] Let $L$ be a pretzel link universe. Then the crossing points in the i-th twist box receive lower labels than those in the j-th box whenever $i \textless j$. In the first twist box, the crossing points are labeled in increasing order from bottom to top, while in each subsequent twist box, the crossing points are labeled in increasing order from top to bottom.  
	\item [(c)] Let $L$ be a Montesinos link universe. Then the unique twist box receives the first k labels. The crossing points in the i-th rational tangle circle receive lower labels than those in the j-th circle whenever $i \textless j$. Within each circle, the crossing points in the i-th twist box receive lower labels than those in the j-th box whenever $i \textless j$. Moreover, in each horizontal twist box, the crossing points are labeled in increasing order from left to right, and in each vertical twist box, they are labeled in increasing order from top to bottom.
\end{itemize}
\end{proposition}
\begin{proof} To prove that these link configurations are indeed admissible, we will show that the first two families can be constructed via the ASI procedure, and that the third family can be obtained from the second by a sequence of admissible bigon extensions.\\
(a) Let $L(a_1+1, a_2,\ldots, a_{n-1}, a_n+1)$ be a 2-bridge link universe, and let $L_{s;l}$ be the associated link configuration constructed according to the prescribed labeling. Then it is clear that $L_{s;l}$ can be constructed via the ASI procedure, with all attached points being of trivial type. More precisely, the sequence of attached points is given by
\[\Big[\underbrace{-p_1,\ldots,-p_{a_1}}_{a_1 \text{points}},\underbrace{+p_{a_1+1},\ldots,+p_{a_1+a_2}}_{a_2 \text{points}},\ldots\Big].\]\\
(b) Let $L(b_1+1, b_2,\ldots, b_{n-1}, b_n+1)$ be a pretzel link universe, and let $L_{s;l}$ be the associated link configuration constructed according to the prescribed labeling. Then $L_{s;l}$ can be constructed via the ASI procedure, and the sequence of attached points is given by
\[\Big[\underbrace{+p_1,\ldots,+p_{b_1}}_{b_1 \text{ points}},\underbrace{-p^C_{b_1+1},\ldots,-p^C_{b_1+n-2}}_{n-2 \text{ points}}\underbrace{+p_{b_1+n-1},\ldots}_{b_n \text{ points}}\Big],\]
where each point $-p^C_{b_1+i}$ is aligned with a curl string consisting of $b_{i+1}-1$ components.\\
(c) We show the Montesinos link universe $L(k;R^1,\ldots,R^m)$ can be obtained from the pretzel link universe 
\[L(\underbrace{1,\ldots,1}_{k \text{ points}},b_{k+1},\ldots,b_{k+m}),\] 
by a sequence of admissible bigon extensions, where  $b_{k+i}=c^i_n+1$ and $(c^i_1,\ldots, c^i_n)$ is the integer array parameterizing the rational tangle $R^i$. Perform bigon extensions of type B at all the top crossing points in the last $m$ twist boxes; see Figure~\ref{fig:extensionM} for reference. Since these crossing points are either active or lie on the shell of the base shell string $S_0$, each of these bigon extensions is admissible.  Note that after the extension, each new crossing point $c$ on the left side receives a lower label than its right counterpart. Therefore, these crossing points are all active. As a result, we can iteratively perform admissible bigon extensions of both types at each such point $c$ to construct the desired rational tangle, as illustrated in Figure~\ref{fig:extensionM}. This completes the proof.
\end{proof}

\begin{figure}[h]
	\centering
	\begin{minipage}[t]{0.55\textwidth}
		\centering	
		\includegraphics[width=1.2\textwidth, height=0.19\textheight]{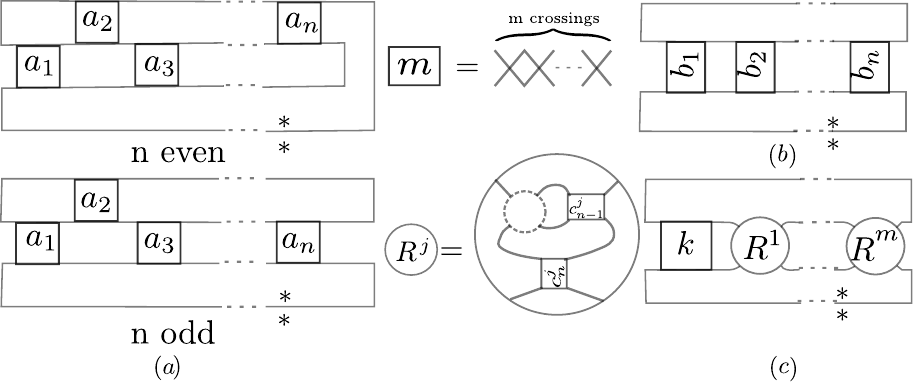}
		\caption{(a) Two-bridge link universes. (b) Pretzel link universes. (c) Montesinos link universes.}
		\label{fig:Links}
	\end{minipage}
	\hfill
	\begin{minipage}[t]{0.44\textwidth}
		\raggedleft
		\includegraphics[width=0.6\textwidth, height=0.19\textheight]{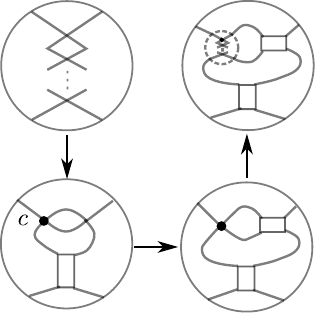}
		\caption{Rational tangles constructed through sequences of admissible bigon extensions.}
		\label{fig:extensionM}
	\end{minipage}
\end{figure}

\section{BIGON REDUCTION AND CLUSTER MUTATION}

In this section, we begin by recalling some basic notation and key results from cluster theory. We then proceed to the proof of the second main theorem. To each \textit{reduction sequence} of an admissible link configuration $L_{s;l}$, we associate a corresponding mutation sequence $\mu_{rd}$ in the cluster algebra. We characterize the cluster variable arising from $\mu_{rd}$, which plays a central role in the proof. By combining this with structural information from $L_{s;l}$, we complete the proof of the theorem. Finally, we illustrate the result with several applications and examples.

\subsection{Background on cluster algebras}

Let N be a positive integer, and let $\operatorname{Trop}(u_1,u_2, \ldots, u_N)$ denote the abelian group freely generated by the elements $u_1,u_2, \ldots, u_N$. The multiplication $\cdot$ in $\operatorname{Trop}(u_1,u_2, \ldots, u_N)$ is defined as the usual multiplication of polynomials and the addition $\oplus$ in $\operatorname{Trop}(u_1,u_2, \ldots, u_N)$ is defined by 
\[\prod^N_{j=1} u_j^{a_j} \oplus \prod^N_{j=1} u_j^{b_j} =
\prod^N_{j=1} u_j^{\min (a_j, b_j)} .\] 
We refer to the structure $(\operatorname{Trop}(u_1,u_2, \ldots, u_N),\oplus,\cdot)$ as a \textit{tropical semifield}. In particular, the tropical semifield $\operatorname{Trop}(y_1, \ldots, y_N)$ is denoted by $\mathbb{P}$, and we let $\mathbb{ZP}$ denote the group ring of $\mathbb{P}$.
The cluster algebra is determined by the choice of an initial seed $\Sigma_0=(\xbf, \ybf ,Q)$, which consists of the following data:
\begin{itemize}
	\item [\textbf{--}] $Q$ is a finite quiver (i.e., a directed graph) with $N$ vertices, containing no loops or 2-cycles;
	\item [\textbf{--}] $\ybf =(y_1,\ldots,y_N)$ is an $N$-tuple called the \textit{initial coeffcient tuple}.
	\item [\textbf{--}] $\xbf =(x_1,\ldots,x_N)$ is an $N$-tuple called the \textit{initial cluster}.
\end{itemize}

The \textit{seed mutation} $\mu_k$ in direction~$k$ transforms
$(\xbf, \ybf, Q)$ into a new seed
$\mu_k(\xbf, \ybf, Q)=(\xbf', \ybf', Q')$, is defined as follows:

\begin{itemize}
	\item The quiver $Q'$ is obtained from $Q$ by performing the following steps: 
	\begin{enumerate}
		\item For every path $i \rightarrow k \rightarrow j$, add one arrow $i \rightarrow j$,
		\item Reverse all arrows incident to $k$,
		\item Remove all resulting 2-cycles.
	\end{enumerate} 
	\item $\ybf'=(y'_1,\ldots,y'_N)$ is a new coefficient $N$-tuple, given by:
	\[y'_j =
	\begin{cases}
		y_k^{-1}, & \text{if $j = k$};\\[.05in]
		y_j \prod_{k\rightarrow j} y_k (y_k\oplus 1)^{-1} \prod_{k\leftarrow j} (y_k\oplus 1),
		& \text{if $j \neq k$},
	\end{cases}\]
	where the first product is taken over all arrows in $Q$ with source $k$, and the second product is taken over all arrows with target $k$.
	\item $\xbf'=(x'_1,\ldots,x'_N)$ is a new cluster $N$-tuple, given by:
	\[x'_j =
	\begin{cases}
     \frac{ y_k\prod_{i\rightarrow k} x_i \ 
     		+ \ \prod_{i \leftarrow k} x_i}{(y_k\oplus 1)x_k}, & \text{if $j = k$};\\[.05in]
	x_k, & \text{if $j \neq k$}.
	\end{cases}\]
\end{itemize}

\begin{definition}
Any variable $x$ that appears in a cluster $\xbf_t$ of a seed $\Sigma_t =(\xbf_t, \ybf_t, Q_t)$ is called a \textit{cluster variable}, where the label $t$ corresponds to a vertex of an $N$-regular tree whose edges represent mutations between seeds. The cluster $N$-tuples and the coefficient $N$-tuples at vertex $t$ are denoted by 
\[\xbf_t=(x_{1;t},\ldots, x_{N;t}), \qquad \ybf_t=(y_{1;t},\ldots, y_{N;t}).\]
The \textit{cluster algebra} $\mathcal{A}=\mathcal{A}(\Sigma_0)$ associated to the initial seed $\Sigma_0$ is defined as 
\[\mathcal{A}=\mathbb{ZP}[x \mid \text{ $x$ is a cluster variable}].\]
\end{definition}

\begin{theorem}\label{thm:compatible}
Let $x_{i;t}$ be a cluster variable, and let $\Sigma_0 =(\xbf_0, \ybf_0, Q_0)$ be the initial seed. Then
\begin{itemize}
	\item [(a)\cite{fomin2002cluster}] The cluster variable can be written as a Laurent polynomial in the initial cluster variables $\xbf_0=(x_1, \dots, x_N)$:
	\begin{equation}\label{eq:Laurent}
		x_{i;t}= \frac{f(x_1, \dots, x_N)}{x_1^{d_1} \cdots x_N^{d_N}},
	\end{equation}
where f is a polynomial that is not divisible by any $x_i$. The exponent vector $(d_1, \dots, d_N)$ of the denominator is called the \textbf{denominator vector} of the cluster variable $x_{i;t}$ with respect to the initial cluster $\xbf_0$.
    \item [(b)\cite{cao2020enough}] Let $\dbf=(d_1, \dots, d_N)$ be the denominator vector of $x_{i;t}$ with respect to the initial cluster $\xbf_0$. Then we have $d_j=0$ if and only if $x_{i;t}$ and $x_j$ both appear in some cluster $\xbf_{t'}$.
\end{itemize}
	\end{theorem}
	
\begin{definition}
Let $x_{i;t}$ be a cluster variable, the \textbf{F-polynomial} $F_{i;t}$ of $x_{i;t}$ is defined as the specializing of $x_{i;t}$ obtained by setting all initial cluster variables $x_j=1$ in its Laurent expansion.
\end{definition}	

\begin{theorem}\cite[Theorem 1.7]{derksen2010quivers}
Each F-polynomial $F_{i;t}$ has constant term 1, and contains a unique monomial of maximal degree with coefficient 1 that is divisible by all other monomials appearing in $F_{i;t}$. In particular, there exists a unique monomial $x_1^{g_1} \cdots x_N^{g_N}$ in the Laurent expansion of $x_{i;t}$ that corresponds to the constant term 1 of $F_{i;t}$. The exponent vector $(g_1, \dots, g_N)$ of this polynomial is called the \textbf{g-vector} of the cluster variable $x_{i;t}$ and is denoted by $\gbf_{i;t}$.
\end{theorem}

\begin{theorem}\cite[Corollary 6.3]{fomin2007cluster}
Let $x_{i;t}$ be a cluster variable, and let $F_{i;t}$ and $\gbf_{i;t}=(g_1,\cdots,g_N)$ be its F-polynomial and g-vector, respectively. Then $x_{i;t}$ can be written as
	\begin{equation}\label{eq:clusterFgformula}
		x_{i;t}= \frac{F_{i;t}(\hat{y}_1,\ldots,\hat{y}_N)}
		{F_{i;t}(y_1,\ldots,y_N)|_{\mathbb{P}}}x_1^{g_1} \ldots x_N^{g_N},
	\end{equation}
where the exchange ratio $\hat{y}_j$ is given by
	\begin{equation}
		\hat{y}_j=y_j \,\frac{\prod_{k \rightarrow j }\, x_{k}}{ \prod_{k \leftarrow j }\, x_{k}}.
	\end{equation}
\end{theorem}

\begin{definition}
Let $x_{i;t}$ be a cluster variable. Then the \textbf{h-vector} $\hbf_{i;t}=(h_1,\cdots ,h_N)$ is defined as the exponent vector of the monomial appearing in the tropical evaluation of the F-polynomial:
\[x^{h_1}_1\cdots x^{h_N}_N=F_{i;t}\big(y_j \mapsto x^{-1}_j\prod_{j \rightarrow k}x_k\big) \big|_{\operatorname{Trop}(x_1\cdots x_N)},\]
where the arrow $j \rightarrow k$ runs over the arrows in the initial quiver  with source $j$.
\end{definition}

\begin{proposition}\cite[Proposition~2.4]{derksen2010quivers}
Let $x_{i;t}$ be a cluster variable obtained from the initial seed $\Sigma_0 =(\xbf_0, \ybf_0, Q_0)$ by a sequence of mutations, with the first mutation performed at k. Let $x^{Q'}_{i;t}$ denote the same variable, but expressed with respect to the initial seed $\mu_k(\Sigma_0)=(\xbf', \ybf', Q')$. Let $\gbf_{i;t}=(g_1, \dots, g_N)$ and $\gbf^{Q'}_{i;t}=(g'_1, \dots, g'_N)$ (resp. $\hbf_{i;t}$ and $\hbf^{Q'}_{i;t}$) be the g-vctor (resp. h-vector) of $x_{i;t}$ and $x^{Q'}_{i;t}$, respectively. Define $b^Q_{jk}=Q_{jk}-Q_{kj}$, where $Q_{jk}$ is the number of arrows from $j$ to $k$ in $Q$. Set

\begin{equation}\label{eq:mutationyseed}
	\ybar_j':= 
	\begin{cases}
	y_k^{-1}, & \text{if } j=k;\\
	y_j(1+y_k)^{b_{jk}}, & \text{if } b_{jk}\geq 0;\\
	y_j (1+y_k^{-1})^{b_{jk}}, & \text{if } b_{jk}\leq 0.
	\end{cases}
\end{equation}
Then the g-vectors are related by
\begin{equation}\label{eq:mutationgvector}
	g'_j=
	\left\{
	\begin{array}
		{ll}
		-g_k, & \textup{if $j=k$} ;\\
		g_j+b_{jk}g_k-b_{jk}h_k, & \textup{if $j \neq k$ and $b_{jk} \textgreater 0$};\\
		g_j-b_{jk}h_k, & \textup{if $j \neq k$ and $b_{jk} \leq 0$}.
	\end{array}
	\right.
\end{equation}
We also have
\begin{equation}
	g_k=h_k-h'_k,
\end{equation}
as well as
\begin{equation}\label{eq:mutationfpoly}
	(1+y_k)^{h_k}F_{l;t}^{Q}(y_1, \dots, y_n) = (1+\ybar_k')^{h'_k} F_{l;t}^{Q'}(\ybar_1',\dots, \ybar_n').
\end{equation}

\end{proposition}

\subsection{From bigon reduction to cluster mutation}
In this subsection, all admissible link configurations are assumed to be constructed via the ASI procedure introduced in Section~4.2. We begin by introducing a quiver associated with each link universe.

\begin{definition}
Let $L_s$ be a link universe with a distinguished segment $s$, and let $G^b_{L,s}$ be the associated balanced overlaid checkerboard graph. The quiver $Q_{L,s}$ is the dual graph of the plane bipartite graph $G^b_{L,s}$, defined as follows:
\begin{itemize}
	\item [\textbf{--}] Place a vertex in each movable face of $G^b_{L,s}$.
	\item [\textbf{--}] For each half-edge $h_{(c,R)}$ of $G^b_{L,s}$ that seperates two distinct movable faces, draw an arrow crossing $h$ such that the crossing vertex $v_c$ is on the left side of the arrow.
	\item [\textbf{--}] Remove all oriented 2-cycles from the resulting directed graph.
\end{itemize}
\end{definition}

\begin{remark}
Since the movable faces correspond to transposable segments, the quiver $Q_{L,s}$ can also be defined directly from $L_s$ by drawing an arrow from a transposable segment $s_b$ to $s_a$ whenever $s_a$ and $s_b$ are incident to the same crossing point $c$, and $s_a$ can be obtained from $s_b$ by a counterclockwise rotation around $c$. For simplicity of notation, we will denote the vertex in $Q$ corresponding to the segment $s_i$ by $i$, and write the variable $y_{s_i}$ as $y_i$.
\end{remark}

The following proposition establishes a connection between the quiver after cluster mutation and the one obtained after bigon reduction.

\begin{proposition}\label{prop:quiverdel}
Let $L_s$ be a link universe with a distinguished segment $s$, and let $D(s_u, s_v)$ denote a bigon reduction performed on $L$, where $s_u$ and $s_v$ are the two reduced segments. A bigon reduction is said to be \textit{permitted} with respect to the segment $s$ if neither $s_u$ nor $s_v$ is $s$. Let $L''$ be the resulting link universe after a permitted bigon reduction. Then the quiver $Q_{L'',s}$ can be obtained from $Q_{L,s}$ by the following steps:
\begin{enumerate}
\item For each vertex in $Q_{L,s}$ corresponding to the square face $f_u$ or $f_v$, perform a mutation at that vertex. Denote the resulting quiver by $Q'$.
\item Delete the vertices at which the mutations were performed, along with all arrows in $Q'$ incident to them.
\end{enumerate}
\end{proposition}
\begin{proof}
See subfigure (a) of Figure~\ref{fig:quiverreduce} for an illustration. Since the bigon reduction is assumed to be permitted, if the segment $s_u$ is not transposable, then the segments $s_a$ and $s_b$ are also not transposable. Therefore, removing the vertices corresponding to untransposable segments, along with all incident arrows, yields the desired result.
\end{proof}

\begin{figure}[hbp]
	\centering
	\subfigure[]{
		\includegraphics[width=0.60\textwidth, height=0.12\textheight]{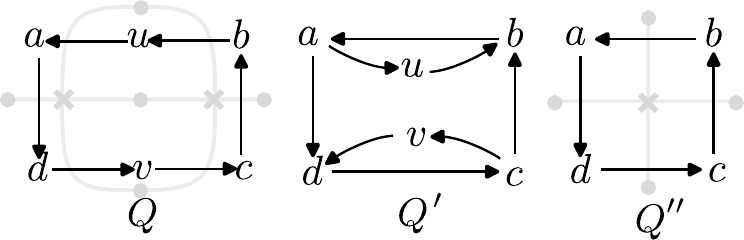}}
	\hspace{2mm}
	\subfigure[]{
		\includegraphics[width=0.35\textwidth,, height=0.12\textheight]{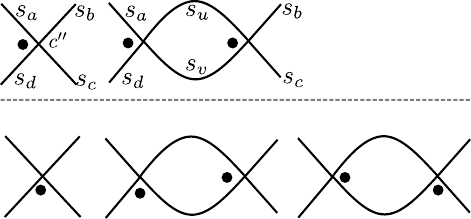}}
	\caption{(a) Subquivers corresponding to local structures arising from bigon reductions. (b) Two types of marker correspondences.}
     \label{fig:quiverreduce}
\end{figure}

\begin{definition}
Let $L_{s;l}$ be an admissible link configuration. Then a \textit{bigon reduction sequence} of $L_{s;l}$ is a set of permitted bigon reductions performed on $L_s$, applied sequentially, such that the resulting link universe is the Hopf link universe. Denote by $s_q$ the unique transposable segment in the Hopf link universe. We represent the sequence by 
\[rd(L_{s;l})\coloneqq s_q, \ldots, s_j, s_k,\]
where, for each bigon reduction $D(s_u, s_v)$ performed, we append the transposable segments in $\{s_u, s_v\}$ to the left side of the sequence, and finally append the segment $s_q$ as the leftmost entry.\\
Associated to each sequence $rd(L_{s;l})$, we define a corresponding mutation sequence $\mu_{rd}$ on the quiver $Q_{L,s}$, given by 
\[\mu_{rd} \coloneqq \mu_q \circ \ldots \circ \mu_j \circ \mu_k.\]
\end{definition}

\vspace{0.5em}

\noindent
After assigning a cluster mutation sequence to each bigon reduction sequence, we characterize the resulting cluster variable in the following theorem. Analogous results for broader classes of links have previously been established in \cite{bazier2024knot} and \cite{meszaros2024dimer}.

\begin{theorem}\label{thm:variable}
Let $L_{s;l}$ be an admissible link configuration. Let $rd(L_{s;l})$ be a bigon reduction sequence of $L_{s;l}$, and let $\mu_{rd}$ be the associated  mutation sequence. Define $\Sigma_t=\mu_{rd}(\Sigma_0)$ to be the seed obtained from the initial seed $\Sigma_0=(\xbf_0,\ybf_0,Q_{L,s})$ by applying the mutation sequence $\mu_{rd}$. Denote by $s_q$ the segment at the leftmost position in the sequence $rd(L_{s;l})$, and let $x_{q;t}$ be the cluster variable in $\xbf_t$ corresponding to $s_q$. Then the following statements hold:
\begin{itemize}
	\item [(a)] Let $\gbf_{q;t}=(g_1,\ldots,g_N)$ be the g-vector of the variable $x_{q;t}$. Then 
	\[g_j=1-k(s_j),\] 
	where $k(s_j)$ denotes the number of markers in the minimal state $\Scal_{\min}$ that are adjacent to the segment $s_j$. In particular, in $\Scal_{\min}$, a counterclockwise transposition can be performed at the segment $s_j$ if and only if $g_j=-1$.
	\item [(b)] Let $F_{q;t}$ be the F-polynomial of the variable $x_{q;t}$. Then
	 \[F_{q;t}(\ybf)= \Mcal_{L,s}(\ybf),\]
	 where $\Mcal_{L,s}(\ybf)$ is the states lattice polynomial of $L_s$.
	 \item [(c)] Let $\dbf_{q;t}=(d_1, \dots, d_N)$ be the denominator vector of the variable $x_{q;t}$. Then we have $d_j=1$ for each $j=1, \dots, N$.
\end{itemize}
\end{theorem}

Before proving Theorem~\ref{thm:variable}, we provide some preliminary results.

\begin{proposition}\cite[Proposition~5.11]{meszaros2024dimer}\label{prop:quiverdelcorr}
The notions are as in Theorem~\ref{thm:variable}. Let $D(s_u, s_v)$ be the first bigon reduction in $rd(L_{s;l})$ to be performed on $L$, and let $L''$ be the resulting link universe. Then $rd'(L_{s;l})$, obtained from $rd(L_{s;l})$ by deleting the reduced segments, is a bigon reduction sequence for $L''_{s;l''}$. Denote the associated mutation sequence by $\mu_{rd'}$. Let $Q'$ be the quiver obtained from $Q_{L;s}$ by performing mutations at the vertices corresponding to the reduced segments, and denote by $Q''$ the quiver $Q_{L'';s}$. Let $x^{Q'}_{q;t}$ and $x^{Q''}_{q;t}$ denote the cluster variables corresponding to the segment $s_q$, obtained by applying the mutation sequence $\mu_{rd'}$ to the initial seeds with initial quivers $Q'$ and $Q''$, respectively. Then the following hold:
\begin{itemize}
	\item [(a)] Let $F^{Q'}_{q;t}$ and $F^{Q''}_{q;t}$ be the F-polynomials corresponding to the variables $x^{Q'}_{q;t}$ and $x^{Q''}_{q;t}$, respectively. Then we have $F^{Q'}_{q;t}=F^{Q''}_{q;t}$.
	\item [(b)] Let $\gbf^{Q'}_{q;t}$ and $\gbf^{Q''}_{q;t}$ be the g-vectors corresponding to the variables $x^{Q'}_{q;t}$ and $x^{Q''}_{q;t}$, respectively. Then the entries of $\gbf^{Q'}_{q;t}$ and $\gbf^{Q''}_{q;t}$ coincide, except for those corresponding to the segments $s_u$ or $s_v$.
\end{itemize}
\end{proposition}

\begin{lemma}\label{lem:monoialdeg}
	Let $L''_s$ be a link universe with a distinguished segment $s$, and let $\Scal''$ be a Kauffman state of $L''_s$. Denote by $m(\Scal'')$ the state monomial associated with $\Scal''$, and let $\deg(y_v)$ denote the exponent of $y_v$ in $m(\Scal'')$. As illustrated in subfigure (b) of Figure~\ref{fig:quiverreduce}, let $R_{dc}$ denote the region adjacent to both segments $s_d$ and $s_c$. Then the following hold:
	\begin{itemize}
		\item [(a)] If $(c'',R_{dc}) \in \Scal''_{\min}$ and $(c'',R_{dc}) \notin \Scal''$, then $\deg(y_c)=\deg(y_d)+1$.
		\item [(b)] If $(c'',R_{dc}) \notin \Scal''_{\min}$ and $(c'',R_{dc}) \in \Scal''$, then $\deg(y_d)=\deg(y_c)+1$.
		\item [(c)] If $(c'',R_{dc})$ belongs to both $\Scal''_{\min}$ and $ \Scal''$, or to neither, then $\deg(y_d)=\deg(y_c)$.
	\end{itemize}
\end{lemma}
\begin{proof}
This follows directly from the construction of the counterclockwise transposition.
\end{proof}

\noindent\textbf{Proof of Theorem~\ref{thm:variable}.}
We proceed by induction on the length of the bigon reduction sequence. Let $H_{s;l_H}$ be a Hopf link configuration, then the quiver $Q_{H,s}$ consists of a single vertex $v_q$ corresponding to the unique transposable segment $s_q$ in $H_s$, and has no arrows. Hence, the mutation sequence $\mu_{rd}$ consists of a single mutation $\mu_q$ at this vertex. The resulting cluster variable is $x_{q;t}=\frac{1+y_q}{x_q}$. Note that in the minimal state $\Scal_{\min}$ of $H_s$, a counterclockwise transposition can be performed at $s_q$, and the states lattice polynomial $\Mcal_{H,s}(\ybf)=1+y_q$. This coincides with the g-vector $\gbf_{q;t}=-1$, the F-polynomial $F_{q;t}=1+y_q$, and the denominator vector $\dbf_{q;t}=(1)$. This completes the verification of the base case. Now, using the notation described in Proposition~\ref{prop:quiverdelcorr}, assume the statement holds for the cluster variable $x^{Q''}_{q;t}$. \\
(a) First, we determine the g-vector $\gbf^{Q'}_{q;t}=(g'_1,\ldots,g'_N)$ of the cluster variable $x^{Q'}_{q;t}$. By part (b) of Proposition~\ref{prop:quiverdelcorr}, it suffices to determine the values of $g'_v$ and $g'_u$. Without loss of generality, assume that at least one of the segments $s_d$ and $s_c$ is transposable, and we address only $g'_v$; the case of $g'_u$ is analogous. By assumption, the F-polynomial $F^{Q''}_{q;t}$ is equal to the states lattice polynomial $\Mcal_{L'',s}$, and by part (a) of Proposition~\ref{prop:quiverdelcorr}, we have $F^{Q'}_{q;t}=F^{Q''}_{q;t}$. Therefore, the tropical evaluation of the F-polynomial satisfies $F^{Q'}_{q;t}(y_1,\ldots,y_n)\big|_{\mathbb{P}}=1$. Now, by Equation~\eqref{eq:clusterFgformula}, the cluster variable $x^{Q'}_{q;t}$ can be written as 
\[ x^{Q'}_{q;t}=F^{Q'}_{q;t}(\hat{y}_1,\ldots,\hat{y}_d,\hat{y}_c,\ldots,\hat{y}_N) \cdot x_1^{g'_1},\ldots,x_v^{g'_v},\ldots,x_N^{g'_N}.\] 
Note that the variable $x_v$ only appears in $\hat y_d$ and $\hat y_c$: specifically, $\hat y_c$ contains a factor of $x_v^{-1}$ while $\hat y_d$ contains $x_v$. By Equation~\eqref{eq:Laurent}, the cluster variable $x^{Q'}_{q;t}$ can be also written as: \[x^{Q'}_{q;t}=\frac{f(x_1, \dots, x_N)}{x_1^{d_1} \cdots x_N^{d_N}},\]
where $f$ is a polynomial not divisible by any of the $x_j$.
Since the mutation sequence $\mu_{rd'}$ does not include $\mu_v$, both $x_v$ and $x^{Q'}_{q;t}$ appear in the same cluster $\xbf^{Q'}_t$. Therefore, by part (b) of Theorem~\ref{thm:compatible}, the variable $x_v$ doesn't appear in the denominator of the above expression.

Now, consider the degrees of $y_d$ and $y_c$ in the monomials of the F-polynomial $F^{Q'}_{q;t}$. Since $F^{Q'}_{q;t}=\Mcal_{L'',s}$, its monomials correspond to the state monomials $m(\Scal'')$ of $L''_s$. By Lemma~\ref{lem:monoialdeg}, we distinguish two cases depending on whether the marker $(c'',R_{dc})$ is in the minimal state $\Scal''_{\min}$:
\begin{itemize}
\item [\text{Case 1.}] If the marker $(c'',R_{dc}) \in \Scal''_{\min}$, we claim the segment $s_c$ is transposable. Suppose, for contradiction, that $s_c$ is not transposable. Then $s_d$ must be transposable, and consequently, the marker $(c'',R_{ad})$ cannot appear in any state of $L''_s$, which contradicts the assumption. Therefore, since $s_d$ is transposable, the marker $(c'',R_{cb})$ must be in some state $\Scal''$. Then by Lemma~\ref{lem:monoialdeg}, there exists a monomial in $F^{Q'}_{q;t}$ satisfying $\deg(y_c)=\deg(y_d)+1$. Substituting into the formula for $x^{Q'}_{q;t}$, this monomial must contain $x_v^{-1}$, which implies $g'_v=1$.
\item [\text{Case 2.}] If the marker $(c'',R_{dc}) \notin \Scal''_{\min}$, the state monomials of $L''_s$ satisfy either $\deg(y_d)=\deg(y_c)+1$ or $\deg(y_d)=\deg(y_c)$. Hence, among the monomials in $F^{Q'}_{q;t}$, the contributions of $x_v$ either cancel or appear only in positive exponents. Therefore, $x_v$ cannot appear in the denominator of $x^{Q'}_{q;t}$, and we conclude $g'_v=0$.
\end{itemize}
Also note that a counterclockwise transposition can be performed at segment $s_v$ in $\Scal_{\min}$ if and only if the marker $(c'',R_{dc})$ is in $\Scal''_{\min}$.

Consider the h-vector $\hbf^{Q'}_{q;t}=(h'_1,\cdots ,h'_N)$ of the cluster variable $x^{Q'}_{q;t}$, We aim to determine the entry $h'_v$. Since $F^{Q'}_{q;t}$ does not contain the variable $y_v$,  in the substitution $y_i =x^{-1}_i\prod_{i \rightarrow j}x_j$, the contributions of $x_v$ either cancel or appear only with positive exponents. Therefore, $x_v$ cannot appear in the monomial obtained after tropical evaluation; we conclude that $h'_v=0$. 

Now, let $\hbf^{Q}_{q;t}=(h_1,\cdots ,h_N)$ be the h-vector of the cluster variable $x^{Q}_{q;t}$. By Equation~\eqref{eq:mutationgvector}, we have $g_v=-g'_v$ and $h_v=g_v$. In addition, the other entries transform as follows: $g_d=g'_d$, $g_c=g'_c-g_v$ and $g_j=g'_j$ for all $j \notin \{v, c, d\}$.
It follows that if the marker $(c'',R_{dc}) \notin \Scal''_{\min}$, then $g_v=0$ and $g_c=g'_c$. And if the marker $(c'',R_{dc}) \in \Scal''_{\min}$, then $g_v=-1$ and $g_c=g'_c+1$. These values align with the number of markers adjacent to the segments $s_v$ and $s_c$ in the minimal states of $L''_s$ and $L_s$, respectively, and thus complete the proof.\\
(b) Since the values of $h_v$ and $h'_v$ have been determined in (a), the F-polynomial $F^{Q}_{q;t}$ can be obtained from $F^{Q''}_{q;t}$ by Equation \eqref{eq:mutationfpoly}. To show that $F^{Q}_{q;t}=\Mcal_{L,s}$, it suffices to demonstrate that the states lattice polynomial $\Mcal_{L,s}$ can be derived from $\Mcal_{L'',s}$ in the same manner. In particular, we focus on the variables $y_d$, $y_v$ and $y_c$, and examine how the local structure near the segment $s_v$ reflects the change in each monomial; the case of $y_b$, $y_u$ and $y_a$ is analogous.

Recall that the F-polynomial $F^{Q}_{q;t}$ is obtained from $F^{Q}_{q;t}$ by substituting 
\[y_c \mapsto \ybar'_c=y_c y_v (y_v+1)^{-1} \text{ and } \, y_d\mapsto \ybar'_d=y_d (y_v+1),\] 
followed by multiplication by 1 if $g_v=0$, or by $1+y_v$ if $g_v=-1$.
\begin{itemize}
\item [\text{Case 1.}] If the marker $(c'',R_{dc}) \in \Scal''_{\min}$, then $g_v=-1$, and the F-polynomial is multiplied by $(1+y_v)$.
\begin{itemize}
\item [(i)]	If a state $\Scal''$ contains the marker $(c'',R_{dc})$, suppose $\deg(y_d)=\deg(y_c)=m$, then \[(1+y_v){\ybar'_d}^m{\ybar'_c}^m=(1+y_v){y_d}^m{y_v}^m{y_c}^m={y_d}^m{y_v}^m{y_c}^m+{y_d}^m{y_v}^{m+1}{y_c}^m.\]
This corresponds to two adjacent states in the lattice: one before the transposition and one after it; see subfigure (b) of Figure~\ref{fig:quiverreduce} for an illustration.
\item [(ii)] For a state $\Scal''$ not containing the marker $(c'',R_{dc})$, then the degrees of $y_d$ and $y_c$ in the state polynomial $m(\Scal'')$ satisfy $\deg(y_d)+1=\deg(y_c)=m+1$, then the substitution gives:
\[(1+y_v){\ybar'_d}^{m}{\ybar'_c}^{m+1}={y_d}^{m}{y_v}^{m+1}{y_c}^{m+1}.\]
Since $(c'',R_{dc}) \in \Scal''_{\min}$, a counterclockwise transposition can be performed at segment $s_v$ in $\Scal_{\min}$. As a result, the degrees of  $y_v$ and $y_c$ in the state monomial are greater than that of $y_d$, consistent with the values obtained via substitution.
\end{itemize}
\item [\text{Case 2.}] If the marker $(c'',R_{dc}) \notin \Scal''_{\min}$, then $g_v=0$, and the F-polynomial is not multiplied further.
\begin{itemize}
	\item [(i)]	 For a state $\Scal''$ containing the marker $(c'',R_{dc})$, suppose $\deg(y_d)=\deg(y_c)+1=m+1$, then
	\[{\ybar'_d}^{m+1}{\ybar'_c}^m=(1+y_v){y_d}^{m+1}{y_v}^m{y_c}^m={y_d}^{m+1}{y_v}^m{y_c}^m+{y_d}^{m+1}{y_v}^{m+1}{y_c}^m.\]
	Note that in this case, a counterclockwise transposition at segment $s_v$ cannot be performed in $\Scal_{\min}$, which is accurately reflected in the degrees of variables in the associated monomials.
    \item [(ii)] For a state $\Scal''$ not containing the marker $(c'',R_{dc})$, suppose $\deg(y_d)=\deg(y_c)=m$, then
       \[{\ybar'_d}^m{\ybar'_c}^m={y_d}^m{y_v}^m{y_c}^m.\]
     In this case, $\Scal''$ corresponds to a single state $\Scal$, in which $y_d$, $y_v$ and $y_c$ have equal degrees, consistent with the values obtained via substitution. This completes the proof.
\end{itemize}
\end{itemize}
(c) From the discussion in the proof of (a), the entries of the denominator $\dbf^{Q'}_{q;t}$ and $\dbf^{Q''}_{q;t}$ agree except for those corresponding to the segments $s_u$ and $s_v$. Moreover, by \cite[Proposition 2.7]{cao2017positivity}, the entries $d_i$ of the denominator $\dbf^{Q}_{q;t}$ coincide with those of $\dbf^{Q'}_{q;t}$ except for $d_u$ and $d_v$. Hence, by the inductive hypothesis, it suffices to show that $d_v=1$; the case of $d_u$ is analogous.\\
Now, by Equation~\eqref{eq:clusterFgformula}, the cluster variable $x^{Q}_{q;t}$ can be written as 
\[ x^{Q}_{q;t}=F^{Q}_{q;t}(\hat{y}_1,\ldots,\hat{y}_d,\hat{y}_c,\ldots,\hat{y}_N) \cdot x_1^{g_1},\ldots,x_v^{g_v},\ldots,x_N^{g_N}.\] 
The variable $x_v$ only appears in $\hat y_d$ and $\hat y_c$: specifically, $\hat y_c$ contains a factor of $x_v$ while $\hat y_d$ contains $x_v^{-1}$. As in the discussion of proof of (a), it can be shown that: if $g_v=-1$, then the contributions of $x_v$ in the monomials of $F^{Q}_{q;t}(\hat{\ybf})$ appear with non-negative exponents; and if $g_v=0$, then there exists a monomial containing $x_v^{-1}$. This ensures $d_v=1$ and completes the proof.
\qed

\begin{lemma}\label{lem:specall}
Let $\Ltil$ be a link diagram, and let $L_{s;l}$ be an admissible link configuration. Then the bracket polynomial $\Gamma_{\Ltil}(A)$ can be expressed as a specialization of the states lattice polynomial of $L_s$. More precisely, we have
\[\Gamma_{\Ltil}(A)=\om(\Scal_{\min})\cdot\Mcal_{L,s}(\ybf)|_{y_i=\om(s_i)},\]
where $\Scal_{\min}$ denotes the minimal state of $L_s$, $\om(\Scal_{\min})$ is its weight, and $\om(s_i)$ is the weight ratio associated to the segment $s_i$.
\end{lemma}
\begin{proof}
By Equation~\eqref{bracketper} and the definition of the weight of the states, the bracket polynomial $\Gamma_{\Ltil}(A)$ can be expressed as:
\[\Gamma=\sum_{\Scal} \om(\Scal)=\om(\Scal_{\min})\sum_{\Scal} \frac{\om(\Scal)}{\om(\Scal_{\min})}.\]
For each Kauffman state $\Scal$, choose a saturated chain $C_{\Scal}$: $\Scal_{\min}=\Scal_0 \lessdot \Scal_1 \cdots \lessdot \Scal_l=\Scal$, and let $s^1\ldots s^l$ denote the transposition words associated to $C_{\Scal}$. Then, by the definition of the weight ratio, we have
\[\frac{\om(\Scal)}{\om(\Scal_{\min})}=\frac{\om(\Scal_1)}{\om(\Scal_{\min})} \cdots \frac{\om(\Scal)}{\om(\Scal_{l-1})}=\om(s^1) \cdots \om(s^l) .\]
On the other hand, the state monomial associated with the state $\Scal$ is
\[m(\Scal)=y_{s^1}\cdots y_{s^l}\]
Recalling that the states lattice polynomial $\Mcal_{L,s}(\ybf)$ is defined as the sum of all such monomials, the desired specialization identity follows.
\end{proof}

\begin{definition}
Let $\Ltil$ be a link diagram, and let $L_{s;l}$ be an admissible link configuration. Suppose that $\Ltil$ is alternating and all of its crossing points have negative signs. Align each transposable segment $s_j$ with a direction such that the black region lies on the left side of $s_j$. Then the \textit{bracket polynomial specialization} is defined by setting
\begin{equation}
	y_j\big|_K=
	\left\{
	\begin{array}
		{ll}
    A^8, & \textup {if $s_j$ goes from the lowest crossing point $c_1$ to a higher one};\\
    -A^4, & \textup {if $s_j$ goes from a lower crossing point (not $c_1$) to a higher one};\\
    -A^{-4}, & \textup {if $s_j$ goes from a higher crossing point to a lower one, and the latter is not active};\\
    A^{-8}, & \textup {if $s_j$ goes from a higher crossing point to a lower one, and the latter is active}.
\end{array}
	\right.
\end{equation}
\end{definition}

We are now ready to prove the second main theorem.

\begin{theorem}\label{thm:fpolyspacial}
Let $\Ltil$ be a link diagram, and let $L_{s;l}$ be an admissible link configuration. Suppose that $\Ltil$ is alternating and all of its crossing points have negative signs.	Let the cluster variable $x^{Q}_{q;t}$ be as defined in Theorem~\ref{thm:variable}. Then the bracket polynomial $\Gamma_{\Ltil}(A)$ can be expressed as a specialization of the F-polynomial $F^{Q}_{q;t}$ via the bracket polynomial specialization. More precisely, we have
\[\Gamma_{\Ltil}(A)=A^{\vert R^u_3\vert+\vert R^l_1\vert-\vert p_-\vert-1}A^{-\vert R^u_1\vert}(-A^3)^{\vert R^l_2\vert+\vert R^l_3\vert}(-A^{-3})^{\vert p_-\vert+1+\vert R^u_2\vert}\cdot F^{Q}_{q;t}(\ybf)\big|_K,\]
where $\vert R^x_i\vert$ denotes the number of certain specific regions as defined in Definition~\ref{def:regions}, and $\vert p_-\vert$ denotes the number of trivial-type attached points on the base string $S_0$ that carry negative signs.
\end{theorem}
\begin{proof}
By part (b) of Theorem~\ref{thm:variable}, we have $F_{q;t}(\ybf)= \Mcal_{L,s}(\ybf)$. By Lemma~\ref{lem:specall}, it therefore suffices to determine the weight of the minimal state $\om(\Scal_{\min})$, and to show that the bracket polynomial specialization of each variable $y_j$ coincides with the corresponding weight ratio $\om(s_j)$.

Since all the crossing points in $\Ltil$ are assumed to have negative signs, the half-edges incident to the round vertices corresponding to white (respectively, black) regions carry positive (respectively, negative) signs. From the activity letters listed in Table~\ref{tab:regionclass} and the discussion in the proof of Proposition~\ref{pro:regiontype}, it follows that the weight of the minimal state is given by
\[\om(\Scal_{\min})=A^{\vert R^u_3\vert+\vert R^l_1\vert-\vert p_-\vert-1}A^{-\vert R^u_1\vert}(-A^3)^{\vert R^l_2\vert+\vert R^l_3\vert}(-A^{-3})^{\vert p_-\vert+1+\vert R^u_2\vert}.\]

Now, consider the weight ratio. Let $s_j$ be a segment incident to two crossing points $c_j$ and $c_k$ with $j \textless k$, as illustrated in Figure~\ref{fig:weightratio}. All possible values of the weight ratio $\om(s_j)$ are also listed there. Observe that if $s_j$ is directed from a lower crossing point to a higher one, then the exponent of $A$ in $\om(s_j)$ is positive; whereas if $s_j$ is directed from a higher point to a lower one, it is negative. Moreover, the absolute value of the exponent is either 4 or 8,  depending on whether the lower crossing point $c_i$ is active. In the case that $c_i$ is active, the absolute value equals 8 precisely when the segment  $s_j$ lies between two half-edges whose associated activity letters are both $L$. Otherwise, it is 4. The active crossing points in $L_{s;l}$ are classified in Corollary~\ref{cor:actclass}. It follows that, for all active crossing points except $c_1$, the segment lying between the two half-edges with activity letter $L$ is always directed from a higher crossing point to the active one. This completes the proof.
\end{proof}

\begin{figure}[hbp]
	\centering
	\subfigure[]{
		\includegraphics[width=0.45\textwidth, height=0.09\textheight]{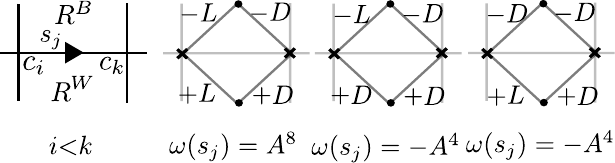}}
	\hspace{4mm}
	\subfigure[]{
		\includegraphics[width=0.45\textwidth, height=0.09\textheight]{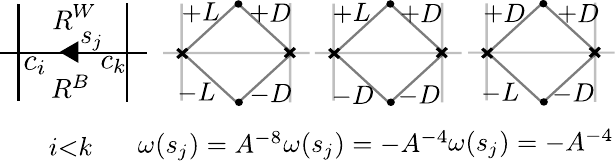}}
	\caption{All possible weight ratios in admissible link configurations.}
	\label{fig:weightratio}
\end{figure}

\begin{remark}
Since the value of the specialization corresponds to the weight ratio, it depends on both the signs of the crossing points and the direction of transposition. If the signs of the crossings are reversed, the exponent in the specialization value will acquire the opposite sign. Similarly, if the maximal state is chosen to be the one with a constant state monomial, the exponent will also change sign accordingly. 
\end{remark}

\begin{corollary}
Let $\Ltil$ be an alternating 2-bridge link diagram such that all of its crossings are negative. Let $L_{s;l}$ denote its associated admissible link configuration. Then its bracket polynomial $\Gamma_{\Ltil}(A)$ can be expressed as a specialization of the F-polynomial $F^{Q}_{q;t}$ via the bracket polynomial specialization
\[\Gamma_{\Ltil}(A)=A^{-\vert R^u_1\vert}(-A^{-3})^{\vert R^l_1\vert}\cdot F^{Q}_{q;t}(\ybf)\big|_K.\]
\end{corollary}
\begin{proof}
Note that in this case, each region belongs to either $R^l_1$ or $R^u_1$, and all attached points on the base string $S_0$ are of trivial type. Therefore, we have $\vert p_-\vert+1 =\vert R^l_1\vert$, and the conclusion follows directly from Theorem~\ref{thm:fpolyspacial}.
\end{proof}

\subsection{Some applications and examples}

The following proposition in fact holds for all links that admit an alternating diagram projection; see \cite{thistlethwaite1987spanning}.

\begin{proposition}
Let $\widetilde{L}$ be an alternating link diagram, and suppose its associated link universe $L$ admits an admissible link configuration. Then the coefficients of its bracket polynomial $\Gamma_{\Ltil}(A)$ alternate in sign with period 4 in the degree.
\end{proposition} 
\begin{proof}
By Theorem~\ref{thm:fpolyspacial}, the bracket polynomial $\Gamma_{\Ltil}(A)$ can be expressed as a specialization of the rescaled states lattice polynomial. In this specialization, each variable corresponds to the weight ratio associated with a transposable segment, and these ratios take the form of powers of $-A^{\pm 4}$. As a result, the coefficients of $\Gamma_{\Ltil}(A)$ alternate in sign with period 4 in the degree.
\end{proof}

\begin{proposition}
Let $\Ltil$ be an alternating 2-bridge link diagram in which all crossings are negative, and let $L_{s;l}$ denote its associated admissible link configuration. Then the monomials of the bracket polynomial $\Gamma_{\Ltil}(A)$ with minimal and maximal degrees correspond to the weights of the minimal and maximal Kauffman states, respectively.
\end{proposition}
\begin{proof}
In this case, the weight ratios take the form of powers of $-A^{4}$. Therefore, the monomials of the bracket polynomial with minimal and maximal degrees correspond to the constant term and the unique highest-degree monomial in the states lattice polynomial, respectively. After rescaling by the weight $\om(\Scal_{\min})$, these become $\om(\Scal_{\min})$ and $\om(\Scal_{\max})$.
\end{proof}

\begin{corollary}
Let $\Ltil(k)$ be an alternating 2-bridge link diagram with one component (or equivalently, a $(2, k)$-torus link) in which all crossings are negative. Then the bracket polynomial of $\Ltil$ is given by:
\[ \Gamma_{\Ltil}(A)=-A^{k-2}+A^{-1}(-A^{-3})^{k-1}\sum^k_{i=0}(-1)^i A^{4i}.\]
\end{corollary}
\begin{proof}
In this case, the set of Kauffman states forms a finite chain lattice with k elements. Moreover, among all the transposable segments, the one incident to the crossing point $c_1$ is the final segment at which a counterclockwise transposition is performed; note that its corresponding variable has specialization value $A^8$. Therefore, the specialization of the states lattice polynomial is given by
\[\mathcal{M}(\ybf)\big|_K=-(-1)^{k-1}A^{4(k-1)}+\sum^k_{i=0}(-1)^i A^{4i}.\]
Also note that $\vert R^u_1\vert=1$ and $\vert R^l_1\vert=k-1$. The formula then follows from a rescaling by $\om(\Scal_{\min})$.
\end{proof}

\begin{figure}[htp]
	\centering
	\includegraphics[width=0.9\textwidth, height=0.08\textheight]{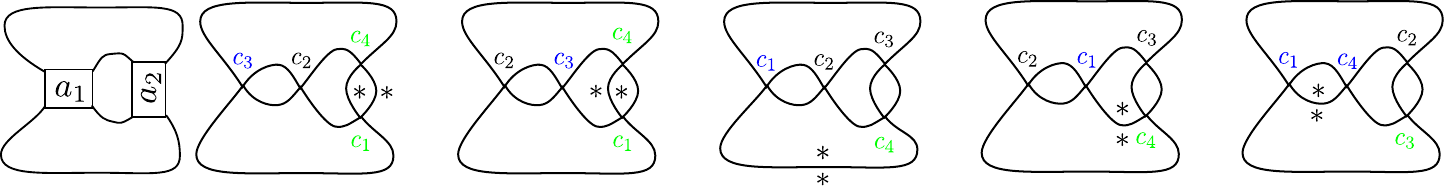}
	\caption{The ordered labelings corresponding to each distinguished segment of the figure-eight link universe.}
	\label{fig:figure8}
\end{figure}
	
\begin{proposition}
Let $L(a_1, a_2)$ be a 2-bridge link universe with two components. Then, for each segment $s$ in $L(a_1, a_2)$, there exists an ordered labeling $l$ such that the associated link configuration $L_{s;l}$ is an admissible link configuration.
\end{proposition}
\begin{proof}
The case where either $a_1=1$ or $a_2=1$ is straightforward. We therefore assume that both $a_1$ and $a_2$ are greater than 1. The various ordered labelings corresponding to each distinguished segment of the figure-eight link universe are displayed in Figure~\ref{fig:figure8}, with segments symmetric about the horizontal axis omitted. It is easy to check that each of them forms an admissible link configuration. Then, by iteratively applying bigon extensions of type A (at the crossings marked in green) and type B (at those marked in blue), we obtain all link configurations of the form $L(a_1, a_2)$ with various distinguished segments. Note that all such crossings are incident to regions that are adjacent to $s$, and therefore the corresponding bigon extensions are admissible. This completes the proof.
\end{proof}

\begin{remark}
A quiver associated to each two-component 2-bridge link universe is constructed and analyzed in detail in \cite{schiffler2022tilting}. However, it is important to note that the bracket polynomial specialization depends on the chosen ordered labeling. Consequently, the same segment may yield different specialization values depending on the distinguished segment relative to which it is evaluated.
\end{remark}

\begin{figure}[htp]
	\centering
	\includegraphics[width=0.83\textwidth, height=0.26\textheight]{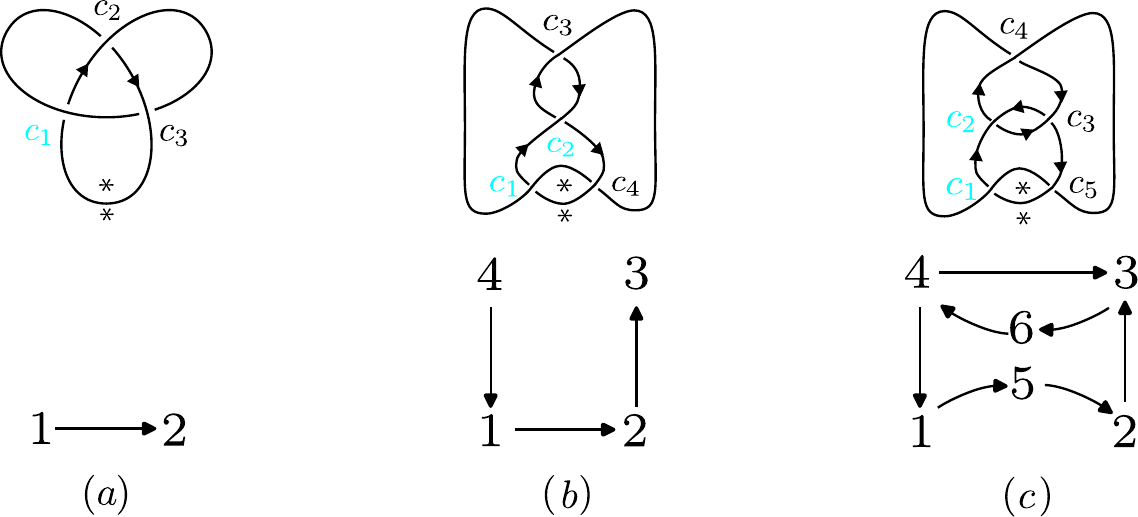}
	\caption{The admissible link configurations and their corresponding quivers, with each transposable segment oriented and active crossings marked in cyan.}
	\label{fig:specialexample}
\end{figure}

\begin{example}
We present three examples demonstrating applications of the specialization formula. The link diagram $\Ltil_{s;l}$ and its associated quiver $Q_{L,s}$ are illustrated in Figure~\ref{fig:specialexample}. The vertices of the quiver are aligned to match the positions of their corresponding transposable segments. Each transposable segment $s_j$ is oriented so that the black region lies to its left, and the active crossing points are marked in cyan.
\begin{itemize}
	\item [(a) ] \textbf{Trefoil:} In this case, $\vert R^u_1\vert=2$ and $\vert R^l_1\vert=1$, hence $\om(\Scal_{\min})=-A^{-5}$. \\
	The mutation sequence is given by $\mu_1 \circ \mu_2$, and we have $F^Q_{1,t}=1+y_1+y_1y_2$. \\
	Therefore, the bracket polynomial is given by
	\[\left\langle \Ltil \right\rangle=\om(\Scal_{\min})\cdot F^Q_{1,t}(\ybf)\big|_{y_1=A^8, \, y_2=-A^4}=-A^{-5}-A^3+A^7.\]
	\item [(b)] \textbf{Figure-eight knot:} In this case, $\vert R^u_1\vert=2$, $\vert R^l_1\vert=1$ and $\vert R^u_2\vert=1$, hence $\om(\Scal_{\min})=A^{-8}$.\\
    The mutation sequence is given by $\mu_1 \circ \mu_2 \circ \mu_3 \circ \mu_4$, and we have 
    \[F^Q_{1,t}=1+y_4+y_1y_4+y_1y_2y_4+y_1y_2y_3y_4.\]
	Therefore, the bracket polynomial is given by
	\[
	\begin{aligned}
		\left\langle \Ltil \right\rangle &=\om(\Scal_{\min}) \cdot F^Q_{1,t}(\ybf)\big|_{y_1=A^8, \, y_2=-A^4, \, y_3=A^{-8}, \, y_4=-A^4}\\ 
		&=A^{-8}-A^{-4}+1-A^4+A^8.
	\end{aligned}
	\]
	\item [(c)] \textbf{Whitehead link:} In this case, $\vert R^u_1\vert=2$, $\vert R^l_1\vert=1$, $\vert R^u_3\vert=1$ and $\vert R^l_3\vert=1$, hence $\om(\Scal_{\min})=A^{-1}$.\\
	The mutation sequence is given by $\mu_1 \circ \mu_2 \circ \mu_3 \circ \mu_4 \circ \mu_5 \circ \mu_6$, and we have
	 \[F^Q_{1,t}=1+y_6+y_4y_6+y_1y_4y_6+y_1y_4y_5y_6+y_1y_2y_4y_5y_6+y_1y_2y_3y_4y_5y_6+y_1y_2y_3y_4y_5y^2_6.\]
	Therefore, the bracket polynomial is given by
	\[
	\begin{aligned}
	\left\langle \Ltil \right\rangle &=\om(\Scal_{\min}) \cdot F^Q_{1,t}(\ybf)\big|_{y_1=A^8, \, y_2=-A^4, \, y_3=-A^{-4}, \, y_4=-A^4, \, y_5=-A^4, \, y_6=A^{-8}} \\ &=A^{-9}-A^{-5}+2A^{-1}-A^3+2A^7-A^{11},
	\end{aligned}
	\]
	which coincides with the formula computed in Section~4 with respect to a different distinguished segment.
\end{itemize}
\end{example}

\bibliographystyle{plain}
\bibliography{KBFpoly.bib}

\end{document}